\documentclass[leqno,11pt]{amsart}
\usepackage{amssymb, amsmath}
\usepackage{color}
\let\pa=\partial
\let\p=\partial
\let\ve=\varepsilon
\let\e=\epsilon
\let\f=\frac
\let\D=\Delta
\let\wt=\widetilde

\def\na{\nabla}
\def\th{\theta}
\def\dive{\mathop{\rm div}\nolimits}
\def\curl{\mathop{\rm curl}\nolimits}

\def\cA{{\mathcal A}}
\def\cB{{\mathcal B}}
\def\cD{{\mathcal D}}
\def\cO{{\mathcal O}}
\def\cM{{\mathcal M}}

\def\dH{\dot{H}}

\def\eqdef{\buildrel\hbox{\footnotesize def}\over =}
\def\eqdefa{\buildrel\hbox{\footnotesize def}\over =}
\def\N{\mathop{\mathbb N\kern 0pt}\nolimits}
\def\Q{\mathop{\mathbb Q\kern 0pt}\nolimits}
\def\R{\mathop{\mathbb R\kern 0pt}\nolimits}
\def\Z{\mathop{\mathbb Z\kern 0pt}\nolimits}

\newcommand{\beq}{\begin{equation}}
\newcommand{\eeq}{\end{equation}}
\newcommand{\ben}{\begin{eqnarray}}
\newcommand{\een}{\end{eqnarray}}
\newcommand{\beno}{\begin{eqnarray*}}
\newcommand{\eeno}{\end{eqnarray*}}
\newcommand{\andf}{\quad\hbox{and}\quad}

\def\ot{\omega^\theta}
\def\wtot{\wt{\omega}^\theta}
\def\ur{u^r}
\def\ut{u^\theta}
\def\uz{u^z}
\def\vr{v^r}
\def\vt{v^\theta}
\def\vz{v^z}
\def\barur{\bar{u}^r}
\def\barut{\bar{u}^\theta}
\def\baruz{\bar{u}^z}
\def\wtur{\wt{u}^r}

\def\wtuz{\wt{u}^z}
\def\dom{\vv\delta}

\def\du{\vv h}

\newtheorem{theorem}{Theorem}[section]
\newtheorem{lemma}{Lemma}[section]
\newtheorem{proposition}{Proposition}[section]

\newtheorem{rmk}{Remark}[section]

\newcommand{\vv}[1]{\boldsymbol{#1}}

\numberwithin{equation}{section}

\begin{document}
\title[Almost axisymmetric Navier-Stokes equations]
{On the existence and structures of almost axisymmetric
solutions to $3$-D Navier-Stokes equations}

\author[Y. Liu]{Yanlin Liu}
\address[Y. Liu]{School of Mathematical Sciences,
Laboratory of Mathematics and Complex Systems,
MOE, Beijing Normal University, 100875 Beijing, China.}
\email{liuyanlin@bnu.edu.cn}

\author[L. Xu]{Li Xu}
\address[L. Xu]{School of Mathematical Sciences, Beihang University, 100191 Beijing, China}
\email{xuliice@buaa.edu.cn}

\date{\today}

\begin{abstract}
In this paper, we consider $3$-D Navier-Stokes equations with almost
axisymmetric initial data, which means that by writing $\vv u_0
=u^r_0\vv e_r+u^\theta_0\vv e_\theta+u^z_0\vv e_z$
in the cylindrical coordinates, then $\pa_\th u^r_0,\,\pa_\th u^\th_0$
and $\pa_\th u^z_0$ are small in some sense
(recall axisymmetric means these three quantities vanish).
Then with additional smallness assumption on $u^\theta_0$,
we prove the global existence of a unique strong solution $\vv u$,
and this solution keeps close to some axisymmetric vector field.
We also establish some refined estimates for the integral average
in $\theta$ variable for $\vv u$.

Moreover, as $u^r_0,\,u^\theta_0$ and $u^z_0$ here depend on $\th$,
it is natural to expand them into Fourier series in $\th$ variable.
And we shall consider one special form of $\vv u_0$,
with some small parameter $\ve$ to measure its swirl part
and oscillating part.
We study the asymptotic expansion of the corresponding solution,
and the influences between different profiles in the asymptotic expansion.
In particular, we give some special symmetric structures that will persist
for all time. These phenomena
reflect some features of the nonlinear terms in Navier-Stokes equations.
\end{abstract}

\maketitle



\setcounter{equation}{0}
\section{Introduction}\label{Secintro}
\subsection{The general setting}
The 3-D incompressible Navier-Stokes equations (N-S) reads:
\beq\label{NS}
\left\{
\begin{aligned}
&\p_t\vv u+\vv u\cdot\nabla\vv u-\Delta\vv u+\nabla P=\vv 0,
\qquad(t,x)\in\R^+\times\R^3\\
&\dive\vv u=0,\\
&\vv u|_{t=0} =\vv u_0,
\end{aligned}
\right.
\eeq
where $\vv u$ stands for the velocity field and $P$
the scalar pressure of the fluid.
This system describes the motion of viscous incompressible fluid.
And it is worth mentioning that except the case for initial data with some special structure,
it is still a big open problem whether or not the system \eqref{NS} has a
unique global solution with large initial data.

Now let us write $\vv u$ in the cylindrical coordinates as
\begin{equation}\label{ucylin}
\vv u(t,x)=u^r(t,r,\th,z)\vv e_r+u^\theta(t,r,\th,z)\vv e_\theta+u^z(t,r,\th,z)\vv e_z,
\end{equation}
where $(r,\theta,z)$ denotes the cylindrical coordinates in $\R^3$
so that $x=(r\cos\theta,r\sin\theta,z)$, and
$$\vv e_r=(\cos\theta,\sin\theta,0),\ \vv e_\theta=(-\sin\theta,\cos\theta,0),\
\vv e_z=(0,0,1),\ r=\sqrt{x_1^2+x_2^2}.$$
Notice that in the cylindrical coordinates, there holds
\beq\label{coordinate}
\na=\vv e_r\p_r+\f{1}{r}\vv e_\theta\p_\theta+\vv e_z\p_z,
\quad \D=\p_r^2+\f{\p_r}{r}+\f{\p_\theta^2}{r^2}+\p_z^2.
\eeq
Then we can reformulate \eqref{NS} in the cylindrical coordinates as
\begin{equation}\label{eqtu}
\left\{
\begin{aligned}
& D_t u^r-\bigl(\pa_r^2+\pa_z^2+\f{\pa_r}{r}
+\f{\pa_\th^2}{r^2}-\f1{r^2}\bigr)u^r
-\f{(u^\theta)^2}{r}
+\f{2\pa_\th u^\theta}{r^2}+\pa_r P=0,\\
& D_t u^\theta-\bigl(\pa_r^2+\pa_z^2+\f{\pa_r}{r}
+\f{\pa_\th^2}{r^2}-\f1{r^2}\bigr)u^\theta
+\f{u^r u^\theta}{r}-\f{2\pa_\th u^r}{r^2}
+\f{\pa_\th P}{r}=0,\\
& D_t u^z-\bigl(\pa_r^2+\pa_z^2+\f{\pa_r}{r}
+\f{\pa_\th^2}{r^2}\bigr) u^z+\pa_z P=0,\\
& \pa_r u^r+\frac{u^r}{r}+\pa_z u^z
+\f{\pa_\th u^\theta}{r}=0,\\
& (u^r,u^\theta,u^z)|_{t=0} =(u_0^r,u^\theta_0,u^z_0),
\end{aligned}
\right.
\end{equation}
where $D_t\eqdefa\pa_t +\vv u\cdot\nabla=\pa_t+\bigl(u^r\pa_r+u^\theta\f{\pa_\th}{r}
+u^z\pa_z\bigr)$ denotes the material derivative.

For the special case when all
the components in \eqref{ucylin} do not depend on $\th$, precisely
\beno
\vv u(t,x)=u^r(t,r,z)\vv e_r+u^\theta(t,r,z)\vv e_\theta+u^z(t,r,z)\vv e_z
\eeno
then this $\vv u$ will be called axisymmetric.
It is a celebrated result that for the more special
axisymmetric without swirl case, which means $u^\theta=0$,
Ladyzhenskaya \cite{La} and independently Ukhovskii
and Yudovich \cite{UY} proved the existence of weak solutions
along with the uniqueness and regularities of such solutions.
And later Ne\v{c}as et al. \cite{LMNP} gave a simpler proof.
Their proofs deeply rely on the fact that
$\omega^\theta/r$ satisfies
\begin{equation}\label{otr}
\pa_t \f{\ot}{r}+(u^r\pa_r+u^z\pa_z) \f{\ot}{r}
-(\Delta+\frac 2r\pa_r)\f{\ot}{r}=0,
\end{equation}
so that any $L^p$ norm of $\omega^\theta/r$ is conserved.

However, for the arbitrary axisymmetric initial data
whose swirl part is non-trivial, we do not have the structure as
that in \eqref{otr} any more.
To the best of our knowledge, so far we can only establish
the global existence of strong solutions
when $\ut_0$ is sufficiently small, and this smallness
needs to depend on other components of $\vv u_0$
or $\nabla\vv u_0$, as the strategy is to view
this case as a small perturbation of the no swirl case.
There are numerous works concerning this situation,
here we only list \cite{CL02, Yau1, Yau2, Chen, Lei, LZ, Wei, ZZT2} for example.
\smallskip

Here we do not limit us in considering axisymmetric solutions,
but this geometric symmetry still plays a crucial role.
Let us also introduce $\bar u(t,x)=\bar u^r(t,r,z) e_r+\barut(t,r,z) e_\th
+\bar u^z(t,r,z) e_z$
satisfying the following axisymmetric N-S:
\begin{equation}\label{eqtbaru}
\left\{
\begin{aligned}
& \pa_t\bar u^r+(\bar u^r\pa_r+\bar u^z\pa_z)\bar u^r
-(\pa_r^2+\pa_z^2+\frac 1r\pa_r-\frac{1}{r^2})\bar u^r
-\frac{(\bar u^\theta)^2}{r}+\pa_r\bar P=0,\\
& \pa_t \bar u^\theta+(\bar u^r\pa_r+\bar u^z\pa_z) \bar u^\theta-(\pa_r^2+\pa_z^2
+\f 1r\pa_r-\f{1}{r^2})\bar u^\theta+\f{\bar u^r\bar u^\theta}{r}=0,\\
& \pa_t\bar u^z+(\bar u^r\pa_r+\bar u^z\pa_z)\bar u^z
-(\pa_r^2+\pa_z^2+\frac 1r\pa_r)\bar u^z+\pa_z\bar P=0,\\
& \pa_r\bar u^r+\frac 1r\bar u^r+\pa_z\bar u^z=0,\\
& \bar u^r|_{t=0}=\barur_0=\cM(u_0^r),\
\bar u^\th|_{t=0}=\barut_0=\cM(u_0^\th),\
\bar u^z|_{t=0}=\baruz_0=\cM(u_0^z),
\end{aligned}
\right.
\end{equation}
where $\cM(f)$ is the integral average of $f$ in $\theta$ variable, precisely
\beq\label{projection}
\cM(f)(r,z)\eqdefa\f{1}{2\pi}\int_0^{2\pi}f(r,\th,z)\,d\th.
\eeq
By definition, there holds $\cM(\p_\th f)=0$ for any smooth function $f$. Then one has
$$\dive\bar{\vv u}_0=\pa_r\bar u^r_0+\frac 1r\bar u^r_0+\pa_z\bar u^z_0
=\cM(\pa_r u^r_0+\frac{u^r_0}{r}+\pa_z u^z_0+\f{\pa_\th u^\theta_0}{r})
=\cM(\dive\vv u_0)=0.$$
This guarantees the initial data in \eqref{eqtbaru} is
indeed compatible with the divergence-free condition.

\subsection{Main results.}\label{subsectionmainthm}
 There are three main results in this paper.
The first one concerns the global well-posedness of \eqref{eqtu}
with initial data that are close to some axisymmetric vector fields in some sense.
To state this precisely, let us introduce the following norms:
\begin{equation}\label{defbarH1}
\|f\|_{\dH^1_{\rm axi}}^2\eqdef\|\wt{\na}f\|_{L^2}^2
+\|f/r\|_{L^2}^2,\andf
\|f\|_{H^1_{\rm axi}}^2\eqdef\|f\|_{L^2}^2+\|f\|_{\dH^1_{\rm axi}}^2,
\end{equation}
where $\wt{\na}\eqdef\vv e_r\p_r+\vv e_z\p_z$ is a part of the whole gradient
given by \eqref{coordinate}.
\begin{theorem}\label{thm1}
{\sl Let $\vv u_0\in H^2$ with $\dive\vv u_0=0$ and $\vv{\mathfrak{U}}_0\eqdef
(\pa_\th u^r_0,\pa_\th u^\theta_0,\pa_\th u^z_0)\in H^1_{\rm axi}$.
If there exists some small positive constant $\e$ such that
the following two smallness conditions hold:
\begin{equation}\label{smallThm1.1}
\|u^\th_0\|_{L^2}\|u^\th_0\|_{\dH^1_{\rm axi}}
\exp\bigl(C\|{\vv u}_0\|_{L^2}^6
\|{\vv u}_0\|_{\dH^2}^2\bigr)<\e,
\end{equation}
\begin{equation}\label{smallThm1.2}
\|\vv{\mathfrak{U}}_0\|_{L^2}\|\vv{\mathfrak{U}}_0\|_{\dot{H}^1_{\rm axi}}
\exp\bigl(\exp(C\|{\vv u}_0\|_{L^2}^6\|{\vv u}_0\|_{\dH^2}^2)\bigr)<\e,
\end{equation}
then both \eqref{eqtu} and \eqref{eqtbaru} have unique global strong solutions $\vv u$ and $\bar{\vv u}$ in $C(\R_+;H^1)\cap L^2(\R_+;\dH^1\cap\dH^2)$ such that for any $t>0$,
there holds
\beq\label{error estimate 1}\begin{aligned}
&\|\vv u-\bar{\vv u}\|_{L^\infty_t(L^2)}^2+\|\nabla(\vv u-\bar{\vv u})\|_{L^2_t(L^2)}^2
\lesssim\|\vv{\mathfrak{U}}_0\|_{L^2}^2
\exp\bigl(\exp(C\|{\vv u}_0\|_{L^2}^6\|{\vv u}_0\|_{\dH^2}^2)\bigr)
\eqdef\cA,\\
&\|\vv u-\bar{\vv u}\|_{L^\infty_t(\dH^1)}^2+\|\nabla(\vv u-\bar{\vv u})\|_{L^2_t(\dH^1)}^2
\lesssim\|\vv{\mathfrak{U}}_0\|_{\dot{H}^1_{\rm axi}}^2
\exp\bigl(\exp(C\|{\vv u}_0\|_{L^2}^6\|{\vv u}_0\|_{\dH^2}^2)\bigr)
\eqdef\cB.
\end{aligned}\eeq
Moreover, for $\cM{\vv u}\eqdefa\cM(u^r)\vv e_r
+\cM(u^\th)\vv e_\th+\cM(u^z)\vv e_z$, there holds
\beq\label{error estimate 2}\begin{aligned}
&\|\cM({\vv u}-\bar{\vv u})\|_{L^\infty_t(L^2)}^2
+\|\nabla\cM({\vv u}-\bar{\vv u})\|_{L^2_t(L^2)}^2
\lesssim\e\cA,\\
&\|\cM({\vv u}-\bar{\vv u})\|_{L^\infty_t(\dH^1)}^2
+\|\nabla\cM({\vv u}-\bar{\vv u})\|_{L^2_t(\dH^1)}^2
\lesssim\e\cB.
\end{aligned}\eeq
}\end{theorem}

\begin{rmk}\label{rmk1.1}
{\rm (i)} We mention that $\vv u_0$ is axisymmetric
when $\vv{\mathfrak{U}}_0$ vanishes, and it is well-known that
in this situation, the corresponding solution
$\vv u$ persists axial symmetry for all time.

In view of this, the smallness condition \eqref{smallThm1.2}
actually tells us that $\vv u_0$ is close to some
axisymmetric vector field. That is what we mean ``almost axisymmetric".
Moreover, it is reasonable to expect that this solution $\vv u$
keeps close to some axisymmetric vector field for all time.

To verify this feature quantitatively, we mention that
one natural way to axisymmetrize a solution $\vv u$ to N-S,
is to axisymmetrize the initial data, and then solve N-S
with this axisymmetrized initial data, in this way we get $\bar{\vv u}$.
Then in view of the smallness condition \eqref{smallThm1.2},
interpolating between the two estimates
in \eqref{error estimate 1} gives
$$\|\vv u-\bar{\vv u}\|_{L^\infty_t(\dH^{\f12})}^2
+\|\nabla(\vv u-\bar{\vv u})\|_{L^2_t(\dH^{\f12})}^2
\lesssim\|\vv{\mathfrak{U}}_0\|_{L^2}\|\vv{\mathfrak{U}}_0\|_{\dot{H}^1_{\rm axi}}
\exp\bigl(\exp(C\|{\vv u}_0\|_{L^2}^6\|{\vv u}_0\|_{\dH^2}^2)\bigr)
<\e,$$
which means that $\vv u$ is indeed close to the axisymmetric vector field $\bar{\vv u}$.

{\rm (ii)} Furthermore, \eqref{error estimate 2} shows that after taking average in $\th$,
$\vv u-\bar{\vv u}$ will become much smaller.
Hence there must be some cancellations in this process,
precisely the positive part of $\vv u-\bar{\vv u}$ almost balance its negative part.
This provides us more details on how $\vv u$ approaches $\bar{\vv u}$.
One can see this more clearly by expanding $\vv u$ into Fourier series
in $\th$ variable, see Remark \ref{rmk1.3} below.
\end{rmk}
\smallskip

With Theorem \ref{thm1} at hand, now let us turn to study the asymptotic expansion
of solutions to N-S with almost axisymmetric initial data.
As we know, a regular enough function can be expanded
into Fourier series in $\th$ variable. And by virtue of the assumptions
of Theorem \ref{thm1}, here we only consider initial data $\vv u(0,x)$
in the following special form
\footnote{In this part, we use $\vv a_0,\,\vv a_k$ and $\vv b_k$ to denote the
profiles of the initial data, while the subscript $0$ is used to denote the
$0$-th Fourier coefficients, so that there would be no confusion.}:
\begin{equation}\label{initialFourier}
\left\{
\begin{aligned}
& \ur(0,x)=a^r_{0}(r,z)+\ve\sum_{k=1}^\infty
\Bigl(a^r_{k}(r,z)\cos k\th+b^r_{k}(r,z)\sin k\th\Bigr),\\
& \ut(0,x)=\ve a^\th_{0}(r,z)+\ve\sum_{k=1}^\infty
\Bigl(a^\th_{k}(r,z)\cos k\th+b^\th_{k}(r,z)\sin k\th\Bigr),\\
& \uz(0,x)=a^z_{0}(r,z)+\ve\sum_{k=1}^\infty
\Bigl(a^z_{k}(r,z)\cos k\th+b^z_{k}(r,z)\sin k\th\Bigr),
\end{aligned}
\right.
\end{equation}
where $\ve>0$ is some small positive constant to be determined later,
and the profiles satisfy
\begin{equation}\begin{split}\label{1.14}
\sum_{j=0}^2\Bigl(\sum_{i=0}^{2-j}\bigl\|r^{-i}\wt{\na}^j
(a^r_{0},a^\th_{0})\bigr\|_{L^2}^2
+\|\wt{\na}^j a_{0}^z\|_{L^2}^2
+\sum_{k=1}^\infty\sum_{i=0}^{2-j} k^{2\max\{i,1\}}
\bigl\|r^{-i}\wt{\na}^j(\vv a_{k},\vv b_{k})\bigr\|_{L^2}^2\Bigr)<\infty,
\end{split}\end{equation}
and
\begin{equation}\label{1.15}
\pa_r a^r_{0}+\f{a^r_{0}}{r}+\pa_z a^z_{0}=0,
\end{equation}
\begin{equation}\label{1.16}
\pa_r a^r_{k}+\f{a^r_{k}}{r}+\pa_z a^z_{k}
+\f{k b^\th_{k}}{r}=0,\quad\pa_r b^r_{k}
+\f{b^r_{k}}{r}+\pa_z b^z_{k}-\f{k a^\th_{k}}{r}=0,\quad\forall\ k\in\N.
\end{equation}

It is not difficult to verify that the constraints \eqref{1.15},~\eqref{1.16}
meet the divergence-free condition $\dive\vv u(0)=0$.
And thanks to Parseval's identity, \eqref{1.14} guarantees that
$\vv u(0)\in H^2$ and $\vv{\mathfrak{U}}(0)\in H^1_{\rm axi}$. Precisely, we have
$$\|\vv{\mathfrak{U}}(0)\|_{L^2}^2\sim\ve^2\sum_{k=1}^\infty k^2
\bigl\|(\vv a_{k},\vv b_{k})\bigr\|_{L^2}^2,\quad
\|\vv{\mathfrak{U}}(0)\|_{\dot{H}^1_{\rm axi}}^2
\sim\ve^2\sum_{k=1}^\infty\sum_{i+j=1}
k^2\bigl\|r^{-i}\wt{\na}^j(\vv a_{k},\vv b_{k})\bigr\|_{L^2}^2,$$
and the other norms can be derived similarly.
In particular, this implies that the smallness conditions in Theorem \ref{thm1}
can be satisfied provided $\ve$ is sufficiently small, and this smallness needs to rely
on the norms of profiles appearing in \eqref{1.14}.

Then Theorem \ref{thm1} guarantees the existence of a unique global strong
solution to N-S with initial data \eqref{initialFourier}.
As we know, a strong solution to N-S would become analytic
for any positive time, thus can be expanded
into Fourier series in $\th$ variable in the following form:
\begin{equation}\begin{split}\label{1.17}
u^\lozenge(t,x)=\sum_{j=0}^\infty\ve^j u^\lozenge_{(j),0}(t,r,z)
+\sum_{j=0}^\infty\sum_{k=1}^\infty
\ve^j\Bigl(u^\lozenge_{(j),k}(t,r,z)\cos k\th
+ v^\lozenge_{(j),k}\sin k\th\Bigr).
\end{split}\end{equation}
where $\lozenge$ can be $r,\,\th$ or $z$,
and the profiles do not rely on $\ve$.

Unlike the Euclidean coordinates,
$\vv e_r$ and $\vv e_\th$ are not constant vectors.
As a result, the convergence of $\vv u$ in Sobolev spaces is
in general not equivalent to the convergence of each component of $\vv u$.
In view of this, it is optimal to verify the validity of the expansion \eqref{1.17}
in $L^\infty(\R^+;L^2\cap L^\infty)$ sense. And
our first result concerning the asymptotic expansion states as follows:

\begin{theorem}\label{thm2}
{\sl Let $\vv u(0,x)$ be given by \eqref{initialFourier} satisfying \eqref{1.14}-\eqref{1.16}.
 Then there exists some $\ve_0>0$ such that for any $\ve\in(0,\ve_0)$,
 \eqref{NS} has a unique global solution $\vv u\in C(\R_+;H^1)\cap
 L^2(\R_+;\dH^1\cap \dH^2)$. Moreover, this solution can be expanded
 as
\begin{equation}\label{1.18}
\left\{
\begin{aligned}
& \ur(t,x)=\ur_{(0),0}(t,r,z)+\ve\sum_{k=1}^\infty
\Bigl(\ur_{(1),k}(t,r,z)\cos k\th+\vr_{(1),k}(t,r,z)\sin k\th\Bigr)+\cO(\ve^2),\\
& \ut(t,x)=\ve\ut_{(1),0}(t,r,z)+\ve\sum_{k=1}^\infty
\Bigl(\ut_{(1),k}(t,r,z)\cos k\th+\vt_{(1),k}(t,r,z)\sin k\th\Bigr)+\cO(\ve^2),\\
& \uz(t,x)=\uz_{(0),0}(t,r,z)+\ve\sum_{k=1}^\infty
\Bigl(\uz_{(1),k}(t,r,z)\cos k\th+\vz_{(1),k}(t,r,z)\sin k\th\Bigr)+\cO(\ve^2)
\end{aligned}
\right.
\end{equation}
in $L^\infty(\R^+;L^2\cap L^\infty)$ sense.
}\end{theorem}

\begin{rmk}
We can see that up to $\ve$ order, this expansion \eqref{1.18} has the same form as the
initial state \eqref{initialFourier}. In particular, the coefficients
$\ur_{(1),0}$ and $\uz_{(1),0}$ vanish for all time.
\end{rmk}

The following result concerns the odevity in this asymptotic expansion.
We mention that this persistence of odevity deeply reflects some nonlinear
structures of N-S.

\begin{theorem}\label{thm3}
{\sl Under the assumptions of Theorem \ref{thm2},
let us consider the special case of \eqref{initialFourier}
that $\ur(0,x)$ and $\uz(0,x)$ are even in $\th$,
while $\ut(0,x)$ is odd in $\th$, precisely:
\begin{equation}\label{inF}
\left\{
\begin{aligned}
& \ur(0,x)=a^r_{0}(r,z)+\ve\sum_{k=1}^\infty
a^r_{k}(r,z)\cos k\th,\\
& \ut(0,x)=\ve\sum_{k=1}^\infty
b^\th_{k}(r,z)\sin k\th,\\
& \uz(0,x)=a^z_{0}(r,z)+\ve\sum_{k=1}^\infty
a^z_{k}(r,z)\cos k\th.
\end{aligned}
\right.
\end{equation}
Then this odevity will persist for all time. Precisely,
the following expansion
\begin{equation}\label{oddexpan}
\left\{
\begin{aligned}
\ur(t,x)=& \Bigl(\ur_{(0),0}+\sum_{j=2}^\infty\ve^j u^r_{(j),0}\Bigr)(t,r,z)
+\sum_{j=1}^\infty\sum_{k=1}^\infty\ve^j \ur_{(j),k}(t,r,z)\cos k\th,\\
\ut(t,x)=&\sum_{j=1}^\infty\sum_{k=1}^\infty\ve^j \vt_{(j),k}(t,r,z)\sin k\th,\\
\uz(t,x)=& \Bigl(\uz_{(0),0}+\sum_{j=2}^\infty\ve^j u^z_{(j),0}\Bigr)(t,r,z)
+\sum_{j=1}^\infty\sum_{k=1}^\infty\ve^j \uz_{(j),k}(t,r,z)\cos k\th
\end{aligned}
\right.
\end{equation}
holds in $L^\infty(\R^+;L^2\cap L^\infty)$ sense.
}\end{theorem}

\begin{rmk}\label{rmk1.3}
One can see from the proof in Section \ref{sec5} that,
 $\ur_{(0),0}\vv e_r+\uz_{(0),0}\vv e_z$ here
actually satisfies axisymmetric N-S with initial data
$a_0^r\vv e_r+a_0^z\vv e_z=\cM(\vv u(0,x))$.
Thus $\ur_{(0),0}\vv e_r+\uz_{(0),0}\vv e_z$ corresponds to the
$\bar{\vv u}$ in Theorem \ref{thm1}, and we can obtain
from the expression \eqref{oddexpan} that
$$\vv u-\bar{\vv u}=\sum_{j=2}^\infty\ve^j\vv u_{(j),0}
+\sum_{j=1}^\infty\sum_{k=1}^\infty\ve^j \Bigl(\vv u_{(j),k}\cos k\th
+\vv v_{(j),k}\sin k\th\Bigr)=\cO(\ve),$$
and
$$\cM(\vv u-\bar{\vv u})=\sum_{j=2}^\infty\ve^j\vv u_{(j),0}
=\cO(\ve^2),$$
which indicates that $\cM(\vv u-\bar{\vv u})$ is in general much smaller
than $\vv u-\bar{\vv u}$, just as what we have mentioned at the end of
Remark \ref{rmk1.1}.
Here we can see that the reason is the cancellations
of the oscillating terms, which is the main terms in $\vv u-\bar{\vv u}$.

Roughly speaking, here we have
$\cM(\vv u-\bar{\vv u})\sim|\vv u-\bar{\vv u}|^2.$
To see this more clearly, we refer the readers to Section \ref{sec5}
to find the derivations of the profiles in \eqref{oddexpan}.
\end{rmk}

Let us end this section with some notations which will be used throughout this paper.

{\bf Notations.}
We shall use $C$ to denote an universal constant which may change from line to line.
The notation $a\lesssim b$ means $a\leq Cb$,
and $f\sim g$ means both $a\lesssim b$ and  $b\lesssim a$ hold.
For a Banach space B, we shall use the shorthand $L^p_T(B)$ for $\bigl\|\|\cdot\|_B\bigr\|_{L^p(0,T;dt)}$.
We use $H^s$ (resp. $\dH^s$) to denote
inhomogeneous (resp. homogeneous) $L^2$ based Sobolev spaces.

\setcounter{equation}{0}
\section{Preliminary}

\subsection{Poincar\'e-type inequality}
\begin{lemma}\label{lemM}
{\sl For any $p\in[1,\infty]$ and $f\in L^p$,
the operator $\cM$ defined in \eqref{projection} satisfies
\begin{equation}\label{MLp}
\|\cM(f)\|_{L^p}\leq\|f\|_{L^p},\andf
\|f-\cM(f)\|_{L^p}\leq2\pi\|\pa_\th f\|_{L^p}.
\end{equation}
}\end{lemma}

\begin{proof}
Let us first consider the case when $p\in[1,\infty)$.
By applying H\"{o}lder's inequality, we get
\begin{align*}
\|\cM(f)\|_{L^p}^p&=2\pi\int_{\R^+\times\R}\Bigl|\f{1}{2\pi}
\int_0^{2\pi}f(r,\th,z)\,d\th\Bigr|^p\,rdrdz\\
&\leq\int_{\R^+\times\R}\bigl(\int_0^{2\pi}|f(r,\th,z)|^p\,d\th\bigr)\,rdrdz
=\|f\|_{L^p}^p.
\end{align*}
To prove the second inequality of \eqref{MLp}, we first write
\begin{align*}
|f(r,\th,z)-\cM(f)(r,z)|^p&=\Bigl|\f{1}{2\pi}
\int_0^{2\pi}\bigl(f(r,\th,z)-f(r,\th',z)\bigr)\,d\th'\Bigr|^p\\
&=\Bigl|\f{1}{2\pi}
\int_0^{2\pi}\bigl(\int_{\th'}^\th(\pa_\th f)(r,\th'',z)\,d\th''\bigr)\,d\th'\Bigr|^p\\
&\leq\f{1}{2\pi}\int_0^{2\pi}\Bigl|\int_{\th'}^\th
\bigl|(\pa_\th f)(r,\th'',z)\bigr|^p\,d\th''\Bigr|
|\th-\th'|^{p-1}\,d\th'\\
&\leq\f{1}{2\pi}\int_0^{2\pi}\bigl|(\pa_\th f)(r,\th'',z)\bigr|^p\,d\th''
\int_0^{2\pi}|\th-\th'|^{p-1}\,d\th'.
\end{align*}
As a result, we can obtain
\begin{align*}
\|f(r,\th,z)-\cM(f)(r,z)\|_{L^p}^p&=\f{1}{2\pi}
\int_{\R^+\times\R}\Bigl(\int_0^{2\pi}\bigl|(\pa_\th f)(r,\th'',z)\bigr|^p
\,d\th''\Bigr)\,rdrdz\\
&\qquad\qquad\qquad\qquad
\times\int_0^{2\pi}\int_0^{2\pi}|\th-\th'|^{p-1}\,d\th'd\th\\
&=\f{2}{p+1}(2\pi)^p\|\pa_\th f\|_{L^p}^p\leq(2\pi)^p\|\pa_\th f\|_{L^p}^p.
\end{align*}
Clearly this gives the second inequality of \eqref{MLp} for $p\in[1,\infty)$.

As for the case when $p=\infty$, we have the following formulas:
$$|\cM(f)(x)|=\Bigl|\f{1}{2\pi}\int_0^{2\pi}f(r,\th,z)\,d\th\Bigr|\leq\|f\|_{L^\infty},$$
and
$$|f(x)-\cM(f)(x)|=\Bigl|\f{1}{2\pi}
\int_0^{2\pi}\bigl(\int_{\th'}^\th(\pa_\th f)(r,\th'',z)\,d\th''\bigr)\,d\th'\Bigr|
\leq2\pi\|\pa_\th f\|_{L^\infty},$$
which hold for any $x\in\R^3$. This completes the proof of this lemma.
\end{proof}

\subsection{Properties for axisymmetric functions.}
Let us first recall the well-known Biot-Savart law,
which asserts that any divergence-free velocity field $\vv u$
can be uniquely determined by its vorticity $\curl\vv u$. Moreover, for any
$s> 0$ and any $1<p<\infty$, there holds
\begin{equation}\label{Biot}
\|\na\vv u\|_{\dH^s}\sim\|\curl\vv u\|_{\dH^s},\quad
\|\nabla\vv u\|_{L^p}\sim\|\curl\vv u\|_{L^p}.
\end{equation}
For the special case when the velocity field is axisymmetric without swirl,
namely $$\wt{\vv u}(x)=\wt{u}^r(r,z)\vv e_r+\wt{u}^z(r,z)\vv e_z,$$
then we can use $\wtot=\p_z\wt u^r-\p_r\wt u^z$ to represent $\widetilde{\vv u}$.
Moreover, there holds:

\begin{lemma}\label{lem2.3}
{\sl If in addition $\wt{\vv u}\in H^2$, then we have
the following estimates:
\begin{equation}\label{lem2.3.1}
\|\na\wt{\vv u}\|_{L^2}
\sim\|\wt{\na}\wt{u}^r\|_{L^2}+\|\wt{\na}\wt{u}^z\|_{L^2}
+\bigl\|\f{\wt{u}^r}{r}\bigr\|_{L^2}\sim\|{\wt\omega}^\theta\|_{L^2},
\end{equation}
\begin{equation}\label{lem2.3.2}
\|\na^2\wt{\vv u}\|_{L^2}^2\sim\|\wt{\na}\wtot\|_{L^2}^2
+\bigl\|\f{\wtot}{r}\bigr\|_{L^2}^2.
\end{equation}
}\end{lemma}
\begin{proof}
Notice that $\wtur,\,\wtuz$ and $\vv e_z$ are independent of $\theta$,
we have the point-wise estimate
$$\bigl|r^{-1}\wtur(x)\bigr|=\bigl|r^{-1}{\pa_\theta}
\bigl(\wt{u}^r(r,z)\vv e_r+\wt{u}^z(r,z)\vv e_z\bigr)\bigr|
\leq|\nabla \wt{\vv u}(x)|,\quad\forall\ x\in\R^3.$$
which along with \eqref{Biot} implies the first desired estimate \eqref{lem2.3.1}.

On the other hand, since $\dive\wt{\vv u}=0$,
$\pa_\theta\vv e_r=\vv e_\theta$ and $\pa_\theta\vv e_\theta=-\vv e_r$, one has
\beno\begin{aligned}
&-\D\wt{\vv u}=\curl\curl\wt{u}=\curl(\wtot\vv e_\theta)=-\p_z\wtot\vv e_r+(\p_r\wtot+\f{\wtot}{r})\vv e_z,
\end{aligned}\eeno
which implies
\beq\label{E 4}\begin{aligned}
\|\na^2\wt{\vv u}\|_{L^2}^2
&\sim\bigl\|\p_r\wtot+\f{\wtot}{r}\bigr\|_{L^2}^2+\|\p_z\wtot\|_{L^2}^2\\
&=\|\wt{\na}\wtot\|_{L^2}^2+\bigl\|\f{\wtot}{r}\bigr\|_{L^2}^2
+2\int_{\R^3}\p_r\wtot\cdot\f{\wtot}{r}\,dx.
\end{aligned}\eeq
While the condition $\wt{\vv u}\in H^2$ implies $\wtot\in L^2$
and $r^{-1}\wtot\in L^2$, and thus
$\lim_{r\rightarrow0}\wtot=\lim_{r\rightarrow\infty}\wtot=0.$
As a result, there holds
\beno
2\int_{\R^3}\p_r\wtot\cdot\f{\wtot}{r}\,dx
=2\pi\int_{\R^+\times\R}\pa_r|\wtot|^2\,drdz
=2\pi\int_{\R}|\wtot|^2\big|_{r=0}^\infty\, dz=0,
\eeno
which along with \eqref{E 4} gives the second desired estimate \eqref{lem2.3.2}. It completes the proof of the lemma.
\end{proof}

\begin{lemma}\label{lembaru}
{\sl Let $\bar{\vv u}(x)=\bar{u}^r(r,z)\vv e_r+\barut(r,z)\vv e_\th
+\bar{u}^z(r,z)\vv e_z\in H^1$ be divergence-free, and the corresponding
vorticity $\bar\omega^\th=\p_z\bar u^r-\p_r\bar u^z$. Then we have
\beq\label{A 5}\begin{aligned}
\|\na\bar{\vv u}\|_{L^2}^2&\sim\|\bar\omega^\th\|_{L^2}^2
+\|\wt{\na}\barut\|_{L^2}^2
+\bigl\|\f{\barut}{r}\bigr\|_{L^2}^2\\
&\sim\|\wt{\na}\bar{u}^r\|_{L^2}^2+\|\wt{\na}\bar{u}^z\|_{L^2}^2+\|\wt{\na}\barut\|_{L^2}^2
+\bigl\|\f{\bar{u}^r}{r}\bigr\|_{L^2}^2+\bigl\|\f{\barut}{r}\bigr\|_{L^2}^2,
\end{aligned}\eeq
}\end{lemma}
\begin{proof}
Notice that $\curl\bar{\vv u}=(-\p_z\barut)\vv e_r
+(\p_z\bar u^r-\p_r\bar u^z)\vv e_\th+(\p_r\barut+\f{\barut}{r})\vv e_z$,
we have
\beq\label{e 3}\begin{aligned}
\|\na\bar{\vv u}\|_{L^2}^2\thicksim\|\curl\bar{\vv u}\|_{L^2}^2
&=\|\bar{\omega}^\th\|_{L^2}^2
+\bigl\|\p_r\barut+\f{\barut}{r}\bigr\|_{L^2}^2+\|\p_z\barut\|_{L^2}^2\\
&\sim\|\wt{\na}\bar{u}^r\|_{L^2}^2+\|\wt{\na}\bar{u}^z\|_{L^2}^2
+\bigl\|\f{\bar{u}^r}{r}\bigr\|_{L^2}^2
+\bigl\|\p_r\barut+\f{\barut}{r}\bigr\|_{L^2}^2+\|\p_z\barut\|_{L^2}^2,
\end{aligned}\eeq
where we have used Lemma \ref{lem2.3} in the last step.
On the other hand, since $\bar{\vv u}\in H^1$, we have
$\barut\in L^2$ and $r^{-1}\barut\in L^2$, and thus
$\lim_{r\rightarrow0}\barut=\lim_{r\rightarrow\infty}\barut=0$,
which implies
\beno
2\int_{\R^3}\p_r\barut\cdot\f{\barut}{r}\,dx
=2\pi\int_{\R^+\times\R}\int_0^\infty\p_r|\barut|^2\,drdz
=2\pi\int_{\R}|\barut|^2\big|_{r=0}^\infty\, dz=0.
\eeno
As a result, there holds
\beno\begin{aligned}
&\bigl\|\p_r\barut+\f{\barut}{r}\bigr\|_{L^2}^2
=\|\p_r\barut\|_{L^2}^2+\bigl\|\f{\barut}{r}\bigr\|_{L^2}^2
+2\int_{\R^3}\p_r\barut\cdot\f{\barut}{r}\,dx
=\|\p_r\barut\|_{L^2}^2+\bigl\|\f{\barut}{r}\bigr\|_{L^2}^2.
\end{aligned}\eeno
Substituting this into \eqref{e 3} gives the desired estimate \eqref{A 5}. The lemma is proved.
\end{proof}

\subsection{A stability result for N-S}
In this subsection, we shall give a stability result for N-S.
There are numerous works concerning this problem,
here we only list two classical results \cite{GDP, Ponce}.

\begin{proposition}\label{stabH1}
{\sl Let $\vv v$ be a global strong solution to N-S
with initial data $\vv v_0\in H^1$, and there exists some positive constants
$A$ and $B$ such that
\begin{equation}\label{defAB}
\|\vv v\|_{L^\infty(\R^+;L^2)}+\|\nabla\vv v\|_{L^2(\R^+;L^2)}\leq A,\andf
\|\vv v\|_{L^\infty(\R^+;\dH^1)}+\|\nabla\vv v\|_{L^2(\R^+;\dH^1)}\leq B.
\end{equation}
Then there exists some small positive constant $c_0$ such that whenever
$\vv u_0\in H^1$ satisfies
\begin{equation}\label{smalldelta}
\|\vv u_0-\vv v_0\|_{L^2}\|\vv u_0-\vv v_0\|_{\dH^1}
\exp\bigl(CA^2B^2\bigr)<c_0,
\end{equation}
then N-S with $\vv u_0$ as initial data
also has a global strong solution $\vv u$. Moreover, this solution satisfies
for $s=0,1$ that
\begin{equation}\label{ineqstabL2H1}
\|\vv u-\vv v\|_{L^\infty(\R^+;\dH^s)}^2+\|\nabla(\vv u-\vv v)\|_{L^2(\R^+;\dH^s)}^2
\leq C\|\vv u_0-\vv v_0\|_{\dH^s}^2\exp\bigl(CA^2B^2\bigr),
\end{equation}
and
\begin{equation}\label{ineqLinfty}
\|\vv u-\vv v\|_{L^\infty(\R^+;L^\infty)}
\leq\|\vv u_0-\vv v_0\|_{H^2}\exp\bigl(CA^2B^2+CB^4\bigr).
\end{equation}
}\end{proposition}

\begin{proof}
{\bf Step 1. The proof of \eqref{ineqstabL2H1}.}
Let us denote $\du\eqdef\vv u-\vv v$, then we can find
\begin{equation}\label{eqtdu}
\left\{\begin{aligned}
& \pa_t\du+\vv u\cdot\nabla\du+\du\cdot\nabla\vv v
-\D\du+\nabla\Pi=0,\quad \dive\du=0,\\
& \du|_{t=0}=\vv u_0-\vv v_0,
\end{aligned}\right.
\end{equation}
for some properly chosen $\Pi$.
By taking $L^2$ inner product of \eqref{eqtdu} with $\du$, we obtain
\begin{align*}
\f12\f{d}{dt}\|\du\|_{L^2}^2+\|\nabla\du\|_{L^2}^2
&=-\int_{\R^3}(\du\cdot\nabla\vv v)\cdot\du\,dx\\
&\leq\|\du\|_{L^4}^2\|\nabla\vv v\|_{L^2}\\
&\leq\f12\|\nabla\du\|_{L^2}^2
+C\|\du\|_{L^2}^2\|\nabla\vv v\|_{L^2}^4.
\end{align*}
After subtracting $\f12\|\nabla\du\|_{L^2}^2$ on both sides,
then applying Gronwall's inequality leads to
\begin{equation}\begin{split}\label{estiduL2}
\|\du\|_{L^\infty_t(L^2)}^2+\|\nabla\du\|_{L^2_t(L^2)}^2
&\leq\|\vv u_0-\vv v_0\|_{L^2}^2
\exp\bigl(C\|\nabla\vv v\|_{L^2_t(L^2)}^2\|\vv v\|_{L^\infty_t(\dH^1)}^2\bigr)\\
&\leq\|\vv u_0-\vv v_0\|_{L^2}^2\exp\bigl(CA^2B^2\bigr).
\end{split}\end{equation}

While by taking $L^2$ inner product of \eqref{eqtdu} with $-\D\vv h$, we have
\beno\begin{aligned}
\f12\f{d}{dt}\|\na\vv h\|_{L^2}^2+\|\D\vv h\|_{L^2}^2
&=(\vv u\cdot\na\vv h\,|\,\D\vv h)+(\vv h\cdot\na\vv v\,|\,\D\vv h)\\
&=-\sum_{j=1}^3(\p_j(\vv v+\vv h)\cdot\na\vv h\,|\,\p_j\vv h)+(\vv h\cdot\na\vv v\,|\,\D\vv h)\\
&\lesssim\|\na\vv h\|_{L^2}^{\f32}\|\na\vv h\|_{L^6}^{\f32}
+\|\na\vv v\|_{L^2}\bigl(\|\na\vv h\|_{L^4}^2+\|\vv h\|_{L^\infty}\|\D\vv h\|_{L^2}\bigr)\\
&\lesssim\|\vv h\|_{L^2}^{\f12}\|\na\vv h\|_{L^2}^{\f12}\|\D\vv h\|_{L^2}^2
+\|\na\vv v\|_{L^2}\|\na\vv h\|_{L^2}^{\f12}\|\D\vv h\|_{L^2}^{\f32}.
\end{aligned}\eeno
Then by using Young's inequality to the last term, we get
\beq\label{B 7}\begin{split}
\f{d}{dt}\|\na\vv h\|_{L^2}^2+\f32\|\D\vv h\|_{L^2}^2
\leq C\|\vv h\|_{L^2}^{\f12}\|\na\vv h\|_{L^2}^{\f12}\|\D\vv h\|_{L^2}^2+C\|\na\vv v\|_{L^2}^4\|\na\vv h\|_{L^2}^2.
\end{split}\eeq

On the other hand, by the local well-posedness result, the following set is not empty:
$$\cD\eqdef\bigl\{\,T\in\R^+:\mbox{ there holds }
\|\nabla\du\|_{L^\infty_T(L^2)}^2+\|\D\du\|_{L^2_T(L^2)}^2
\leq2\|\nabla(\vv u_0-\vv v_0)\|_{L^2}^2
\exp\bigl(CA^2B^2\bigr)\,\bigr\}.$$
Let us take $T^*=\sup\{T:T\in\cD\}$. If $T^*<\infty$,
then for any $t\leq T^*$, we can use the estimate \eqref{estiduL2}
and the smallness condition \eqref{smalldelta} with $c_0$ sufficiently small to deduce
\begin{align*}
C\|\du\|_{L^\infty_t(L^2)}^{\f12}\|\nabla\du\|_{L^\infty_t(L^2)}^{\f12}
\leq C\|\vv u_0-\vv v_0\|_{L^2}^{\f12}\|\nabla(\vv u_0-\vv v_0)\|_{L^2}^{\f12}
\exp\bigl(CA^2B^2\bigr)\leq Cc_0^{\f12}<\f12.
\end{align*}
By substituting this into \eqref{B 7},
and then using Gronwall's inequality, we achieve
\begin{align*}
\|\nabla\du(t)\|_{L^2}^2+\|\D\du\|_{L^2_t(L^2)}^2
&\leq\|\nabla(\vv u_0-\vv v_0)\|_{L^2}^2
\exp\bigl(C\|\nabla\vv v\|_{L^\infty_t(L^2)}^2
\|\nabla\vv v\|_{L^2_t(L^2)}^2\bigr)\\
&\leq\|\nabla(\vv u_0-\vv v_0)\|_{L^2}^2
\exp\bigl(CA^2B^2\bigr),
\end{align*}
which contradicts to the definition of $T^*$. Thus there
must be $T^*=\infty$, and we have
\begin{equation}\label{estidomL2}
\|\nabla\du\|_{L^\infty(\R^+;L^2)}^2+\|\D\du\|_{L^2(\R^+;L^2)}^2
\leq2\|\nabla(\vv u_0-\vv v_0)\|_{L^2}^2
\exp\bigl(CA^2B^2\bigr).
\end{equation}

\smallskip

{\bf Step 2. The proof of \eqref{ineqLinfty}.}
Let us denote $\dom\eqdef\curl \du$
and $\vv \omega\eqdef\curl\vv v$, which satisfy
\begin{equation}\label{eqtdo}
\left\{\begin{split}
& \pa_t\dom+\vv u\cdot\nabla\dom+\du\cdot\nabla\vv \omega
-\D\dom=\dom\cdot\nabla\vv u+\vv\omega\cdot\nabla\du,\\
& \dom|_{t=0}=\curl(\vv u_0-\vv v_0).
\end{split}\right.
\end{equation}
By taking $L^2$ inner product of \eqref{eqtdo} wit $|\dom|^2\dom$, we get
\begin{equation}\label{estidomL4}
\f14\f{d}{dt}\|\dom\|_{L^4}^4+\bigl\||\dom|\nabla\dom\bigr\|_{L^2}^2
+\f12\bigl\|\nabla|\dom|^2\bigr\|_{L^2}^2
\leq\Bigl|\int_{\R^3}(\dom\cdot\nabla\vv u+\vv\omega\cdot\nabla\du
-\du\cdot\nabla\vv\omega)\cdot|\dom|^2\dom\,dx\Bigr|.
\end{equation}
The terms on the right-hand side can be handled as follows:
\begin{align*}
\Bigl|\int_{\R^3}(\dom\cdot\nabla\vv u)\cdot|\dom|^2\dom\,dx\Bigr|
&\leq\bigl\||\dom|^2\bigr\|_{L^3}\bigl\||\dom|^2\bigr\|_{L^6}
\|\nabla(\vv v+\du)\|_{L^2}\\
&\leq\f18\bigl\|\nabla|\dom|^2\bigr\|_{L^2}^2
+C\bigl\||\dom|^2\bigr\|_{L^2}^2\bigl(\|\vv\omega\|_{L^2}^4+\|\nabla\du\|_{L^2}^4\bigr),
\end{align*}
and by using the Biot-Savart law that
\begin{align*}
\Bigl|\int_{\R^3}(\vv\omega\cdot\nabla\du)\cdot|\dom|^2\dom\,dx\Bigr|
&\leq\|\vv\omega\|_{L^2}\|\nabla\du\|_{L^6}
\bigl\||\dom|^2\bigr\|_{L^6}\|\dom\|_{L^6}\\
&\leq C\|\vv\omega\|_{L^2}\bigl\||\dom|^2\bigr\|_{L^3}
\bigl\||\dom|^2\bigr\|_{L^6}\\
&\leq\f18\bigl\|\nabla|\dom|^2\bigr\|_{L^2}^2
+C\bigl\||\dom|^2\bigr\|_{L^2}^2\|\vv\omega\|_{L^2}^4,
\end{align*}
and by using integration by parts together with the divergence-free
condition that
\begin{align*}
\Bigl|\int_{\R^3}(&\du\cdot\nabla\vv\omega)\cdot|\dom|^2\dom\,dx\Bigr|
=\Bigl|\int_{\R^3}(\vv\omega\otimes\du):\nabla\bigl(|\dom|^2\dom\bigr)\,dx\Bigr|\\
&\leq C\bigl(\bigl\|\nabla|\dom|^2\bigr\|_{L^2}
+\bigl\||\dom|\nabla\dom\bigr\|_{L^2}\bigr)\|\dom\|_{L^{12}}
\|\du\|_{L^{12}}\|\vv\omega\|_{L^3}\\
&\leq C\bigl(\bigl\|\nabla|\dom|^2\bigr\|_{L^2}
+\bigl\||\dom|\nabla\dom\bigr\|_{L^2}\bigr)\bigl\|\nabla|\dom|^2\bigr\|_{L^2}^{\f12}
\|\du\|_{L^4}^{\f12}\|\nabla \du\|_{L^4}^{\f12}
\|\vv\omega\|_{L^2}^{\f12}\|\nabla\vv\omega\|_{L^2}^{\f12}\\
&\leq\f18\bigl(\bigl\|\nabla|\dom|^2\bigr\|_{L^2}^2
+\bigl\||\dom|\nabla\dom\bigr\|_{L^2}^2\bigr)
+C\bigl(\|\du\|_{L^4}^4+\|\nabla\du\|_{L^4}^4\bigr)
\|\vv\omega\|_{L^2}^2\|\nabla\vv\omega\|_{L^2}^2.
\end{align*}
Now by substituting the above three estimates into \eqref{estidomL4},
and using the relation that
$$\|\nabla\du\|_{L^4}\lesssim\|\curl\du\|_{L^4}=\|\dom\|_{L^4},$$
we get
\begin{align*}
\f{d}{dt}\|\dom\|_{L^4}^4+\bigl\||\dom|\nabla\dom\bigr\|_{L^2}^2
+\bigl\|\nabla|\dom|^2\bigr\|_{L^2}^2
\lesssim\|\dom\|_{L^4}^4\bigl(\|\vv\omega\|_{L^2}^2\|\vv\omega\|_{H^1}^2
+\|\nabla\du\|_{L^2}^4\bigr)
+\|\du\|_{L^4}^4\|\vv\omega\|_{L^2}^2\|\nabla\vv\omega\|_{L^2}^2.
\end{align*}
Then Gronwall's inequality together with the bounds
\eqref{defAB} and \eqref{ineqstabL2H1} leads to
\begin{equation}\begin{split}\label{estidomL4}
\|\dom\|_{L^\infty_t(L^4)}^4
\lesssim&\bigl(\|\nabla(\vv u_0-\vv v_0)\|_{L^4}^4
+\|\du\|_{L^\infty_t(L^4)}^4
\|\vv\omega\|_{L^\infty_t(L^2)}^2\|\nabla\vv\omega\|_{L^2_t(L^2)}^2\bigr)\\
&\times\exp\bigl(C\|\vv\omega\|_{L^\infty_t(L^2)}^2\|\vv\omega\|_{L^2_t(H^1)}^2
+C\|\nabla\du\|_{L^\infty_t(L^2)}^2\|\nabla\du\|_{L^2_t(L^2)}^2\bigr)\\
\lesssim&\bigl(\|\vv u_0-\vv v_0\|_{\dH^{\f74}}^4
+B^4\|\vv u_0-\vv v_0\|_{L^2}\|\vv u_0-\vv v_0\|_{\dH^1}^3\bigr)
\exp\bigl(CA^2B^2+CB^4+Cc_0^2\bigr)\\
\lesssim& \|\vv u_0-\vv v_0\|_{H^2}^4
\exp\bigl(CA^2B^2+CB^4\bigr).
\end{split}\end{equation}
Now interpolating between \eqref{estidomL2} and \eqref{estidomL4} gives
\begin{align*}
\|\du\|_{L^\infty_t(L^\infty)}
&\lesssim\|\nabla\du\|_{L^\infty_t(L^2)}^{\f13}
\|\nabla\du\|_{L^\infty_t(L^4)}^{\f23}\\
&\lesssim\|\dom\|_{L^\infty_t(L^2)}^{\f13}
\|\dom\|_{L^\infty_t(L^4)}^{\f23}
\lesssim\|\vv u_0-\vv v_0\|_{H^2}
\exp\bigl(CA^2B^2+CB^4\bigr),
\end{align*}
which is exactly the desired estimate \eqref{ineqLinfty}.
This completes the proof of this proposition.
\end{proof}

\setcounter{equation}{0}

\section{The proof of Theorem \ref{thm1}}

The purpose of this section is to prove Theorem \ref{thm1}
by using perturbation method.

\subsection{Global solvability of \eqref{eqtbaru}}
Let us first introduce the
velocity field $\wt{\vv u}(x)=\wtur(r,z)\vv e_r+\wtuz(r,z)\vv e_z$ satisfying
the following axisymmetric without swirl N-S:
\begin{equation}\label{eqtwtu}
\left\{
\begin{split}
& \pa_t\wt u^r+(\wt u^r\pa_r+\wt u^z\pa_z)\wt u^r
-(\pa_r^2+\pa_z^2+\frac 1r\pa_r-\frac{1}{r^2})\wt u^r+\pa_r\wt P=0,\\
& \pa_t\wt u^z+(\wt u^r\pa_r+\wt u^z\pa_z)\wt u^z
-(\pa_r^2+\pa_z^2+\frac 1r\pa_r)\wt u^z+\pa_z\wt P=0,\\
& \pa_r\wt u^r+\frac 1r\wt u^r+\pa_z\wt u^z=0,\\
& \wt u^r|_{t=0}=\wtur_0=\cM(u_0^r),\
\wt u^z|_{t=0}=\wtuz_0=\cM(u_0^z).
\end{split}
\right.
\end{equation}
\begin{proposition}\label{P 1}
{\sl Under the assumption of Theorem \ref{thm1},
\eqref{eqtwtu} has a unique global solution $\wt{\vv u}$ such that
for any time $t>0$, there hold
\begin{equation}\label{estiwtuL2}
\|\wt{\vv u}\|_{L^\infty_t(L^2)}^2+2\|\nabla\wt{\vv u}\|_{L^2_t(L^2)}^2
\leq\|{\vv u}_0\|_{L^2}^2,
\end{equation}
\begin{equation}\label{estiwtudH1}
\|\wt{\vv u}\|_{L^\infty_t(\dH^1)}^2+\|\nabla\wt{\vv u}\|_{L^2_t(\dH^1)}^2
\lesssim\|{\vv u}_0\|_{\dH^1}^2
+\|{\vv u}_0\|_{L^2}^4
\|{\vv u}_0\|_{\dH^2}^2.
\end{equation}
}\end{proposition}

\begin{proof}
The first estimate \eqref{estiwtuL2} is nothing but the energy equality
and the fact that
$$\|\wt{\vv u}_0\|_{L^2}^2=\|\cM(u_0^r)\|_{L^2}^2+\|\cM(u_0^z)\|_{L^2}^2
\leq\|{\vv u}_0\|_{L^2}^2.$$

Next, the Biot-Savart law tells us that we can use
$\curl\wt{\vv u}=\wtot\vv e_\th=(\pa_z\wtur-\pa_r\wtuz)\vv e_\th$
to represent $\widetilde{\vv u}$.
Hence we can reformulate the System \eqref{eqtwtu} as
\begin{equation}\label{eqtwtot}
\left\{
\begin{split}
& \pa_t \wtot+(\wtur\pa_r+\wtuz\pa_z) \wtot
-(\pa_r^2+\pa_z^2+\frac 1r\pa_r-\frac{1}{r^2})\wtot
-\frac{\wtur\wtot}{r}=0,\\
& \wtot|_{t=0}=\wtot_0=\cM(\pa_z u^r_0-\pa_r u^z_0).
\end{split}
\right.
\end{equation}
Then it is not difficult to verify that $\wtot/r$ satisfies
\begin{equation}\label{eqtwtotr}
\pa_t \f{\wtot}{r}+(\wtur\pa_r+\wtuz\pa_z)\f{\wtot}{r}
-(\Delta+\f 2r\pa_r)\f{\wtot}{r}=0.
\end{equation}

For a strong solution of N-S, the decay property at infinity implies $\wtot|_{r=\infty}=0$,
while the smoothness implies $\wtot|_{r=0}=0$. As a result, we have
$$-\int_{\R^3}\f2r\pa_r\f{\wtot}{r}\cdot\f{\wtot}{r}\,dx
=-2\pi\int_{\R^+\times\R}\pa_r\bigl|\f{\wtot}{r}\bigr|^2\,drdz
=2\pi\int_{\R} \lim_{r\rightarrow0}\bigl|\f{\wtot}{r}\bigr|^2\,dz\geq0.$$
In view of this, the $L^2$ energy estimate of $\wtot/r$ in \eqref{eqtwtotr} gives
\begin{equation}\label{estiwtotr'}
\bigl\|\f{\wtot}r\bigr\|_{L^\infty_t(L^2)}
+\bigl\|\nabla\f{\wtot}r\bigr\|_{L^2_t(L^2)}
\leq\bigl\|\f{\wtot_0}r\bigr\|_{L^2}.
\end{equation}
By using \eqref{Biot} and Lemma \ref{lemM}, we have
\begin{align*}
\|\nabla^2\vv u_0\|_{L^2}\sim\|\nabla\vv \omega_0\|_{L^2}
&\geq\bigl\|\f{\pa_\th}r(\vv\omega_0)\cdot\vv e_r\bigr\|_{L^2}
=\bigl\|-\f{\ot_0}r+\f{\pa_\th\omega^r_0}r\bigr\|_{L^2}\\
&\geq\bigl\|\cM\bigl(-\f{\ot_0}r+\f{\pa_\th\omega^r_0}r\bigr)\bigr\|_{L^2}
=\bigl\|\cM\bigl(-\f{\ot_0}r\bigr)\bigr\|_{L^2}
=\bigl\|\f{\wtot_0}r\bigr\|_{L^2},
\end{align*}
where $\vv \omega_0\eqdefa\curl\vv u_0$.
Substituting this into \eqref{estiwtotr'} gives
\begin{equation}\label{estiwtotr}
\bigl\|\f{\wtot}r\bigr\|_{L^\infty_t(L^2)}
+\bigl\|\nabla\f{\wtot}r\bigr\|_{L^2_t(L^2)}
\leq\|\vv u_0\|_{\dH^2}.
\end{equation}

\smallskip

Next, by taking $L^2$ inner product of \eqref{eqtwtot} with $\wtot$, we obtain
\begin{equation}\label{estiwtotL21}
\f12\f{d}{dt}\|\wtot\|_{L^2}^2+\|\nabla\wtot\|_{L^2}^2+\bigl\|\f{\wtot}{r}\|_{L^2}^2
\leq\int_{\R^3}\f{\wtur}r|\wtot|^2\,dx.
\end{equation}
For the right-hand side term, we have
\begin{align*}
\bigl|\int_{\R^3}\f{\wtur}r|\wtot|^2\,dx\bigr|
&\leq\|\wtur\|_{L^2}\|\wtot\|_{L^6}\bigl\|\f{\wtot}{r}\bigr\|_{L^2}^{\f12}
\bigl\|\f{\wtot}{r}\bigr\|_{L^6}^{\f12}\\
&\leq\f12\bigl(\|\nabla\wtot\|_{L^2}^2+\bigl\|\f{\wtot}{r}\bigr\|_{L^2}^2\bigr)
+C\|\wtur\|_{L^2}^4\|\na\bigl(\f{\wtot}{r}\bigr)\|_{L^2}^2.
\end{align*}
Substituting this into \eqref{estiwtotL21} and then integrating in time,
we deduce
\begin{equation}\label{estiwtotL22}
\|\wtot\|_{L^\infty_t(L^2)}^2+\|\nabla\wtot\|_{L^2_t(L^2)}^2
+\|\f{\wtot}{r}\|_{L^2_t(L^2)}^2
\leq\|\wtot_0\|_{L^2}^2+C\|\wtur\|_{L^\infty_t(L^2)}^4
\|\na\bigl(\f{\wtot}{r}\bigr)\|_{L^2_t(L^2)}^2.
\end{equation}

Thanks to the estimates \eqref{estiwtotr} and \eqref{estiwtuL2}
, we finally dedecue from \eqref{estiwtotL22} that
\begin{align*}
\|\wtot\|_{L^\infty_t(L^2)}^2+\|\nabla\wtot\|_{L^2_t(L^2)}^2
+\|\wtot/r\|_{L^2_t(L^2)}^2
\lesssim\|{\vv u}_0\|_{\dH^1}^2
+\|{\vv u}_0\|_{L^2}^4
\|{\vv u}_0\|_{\dH^2}^2,
\end{align*}
which together with Lemma \ref{lem2.3} leads to the second desired
estimate \eqref{estiwtudH1}.
\end{proof}
\smallskip

Next, we shall use Proposition \ref{stabH1} to derive the
global existence of the solution to \eqref{eqtbaru}.

\begin{proposition}\label{P 2}
{\sl Under the assumption of Theorem \ref{thm1},
\eqref{eqtbaru} has a unique global solution $\bar{\vv u}$ such that
for any time $t>0$, there hold
\begin{equation}\label{estibaruL2}
\|\bar{\vv u}\|_{L^\infty_t(L^2)}^2+2\|\nabla\bar{\vv u}\|_{L^2_t(L^2)}^2
\leq\|{\vv u}_0\|_{L^2}^2
,
\end{equation}
\begin{equation}\label{estibaruH1}
\|\bar{\vv u}\|_{L^\infty_t(\dH^1)}^2+\|\nabla\bar{\vv u}\|_{L^2_t(\dH^1)}^2
\lesssim\|\na\vv u_0\|_{L^2}^2
\exp\bigl(C\|{\vv u}_0\|_{L^2}^6\|{\vv u}_0\|_{\dH^2}^2\bigr).
\end{equation}
}\end{proposition}

\begin{proof}
By using Lemma \ref{lemM} and the fact that $\bar{\vv u}_0-\wt{\vv u}_0
=\barut_0(r,z)\vv e_\th=\cM(\ut_0)\vv e_\th$, we have
\beno\|\bar{\vv u}_0-\wt{\vv u}_0\|_{L^2}=\|\cM(\ut_0)\|_{L^2}
\leq\|\ut_0\|_{L^2},
\eeno
and
\begin{align*}
\|\na(\bar{\vv u}_0-\wt{\vv u}_0)\|_{L^2}^2
&=\bigl\|\na\bigl(\cM(\ut_0)\vv e_\th\bigr)\bigr\|_{L^2}^2
=\|\wt{\na}\cM(\ut_0)\|_{L^2}^2+\bigl\|\f{\cM(\ut_0)}{r}\bigr\|_{L^2}^2\\
&=\|\cM(\wt{\na}\ut_0)\|_{L^2}^2
+\min\Bigl\{\bigl\|\cM\bigl(\f{\ut_0}{r}\bigr)\bigr\|_{L^2}^2
,\bigl\|\cM\bigl(\f{\ut_0-\pa_\th\ur_0}{r}\bigr)\bigr\|_{L^2}^2\Bigr\}\\
&\leq\|\wt{\na}\ut_0\|_{L^2}^2
+\min\Bigl\{\bigl\|\f{\ut_0}{r}\bigr\|_{L^2}^2
,\bigl\|\f{\ut_0-\pa_\th\ur_0}{r}\bigr\|_{L^2}^2\Bigr\}\\
&\leq\min\bigl\{\|u^\th_0\|_{\dot{H}^1_{\rm axi}}^2,
\|{\vv u}_0\|_{\dH^1}^2\bigr\}.
\end{align*}
These together with the bounds \eqref{estiwtuL2},~\eqref{estiwtudH1}
and the smallness condition \eqref{smallThm1.1},
we get that
\beno\begin{aligned}
&\|\bar{\vv u}_0-\wt{\vv u}_0\|_{L^2}\|\bar{\vv u}_0-\wt{\vv u}_0\|_{\dH^1}\exp\bigl(C\|{\vv u}_0\|_{L^2}^2(\|{\vv u}_0\|_{\dH^1}^2
+\|{\vv u}_0\|_{L^2}^4\|{\vv u}_0\|_{\dH^2}^2)\bigr)\\
&\qquad\leq C\|\ut_0\|_{L^2}\|u^\th_0\|_{\dot{H}^1_{\rm axi}}\exp\bigl(C\|{\vv u}_0\|_{L^2}^6\|{\vv u}_0\|_{\dH^2}^2\bigr)<c_0,
\end{aligned}\eeno
where we used the fact that
\beno
\|{\vv u}_0\|_{L^2}^2\|{\vv u}_0\|_{\dH^1}^2\lesssim\|{\vv u}_0\|_{L^2}^3\|{\vv u}_0\|_{\dH^2}\lesssim 1+\|{\vv u}_0\|_{L^2}^6\|{\vv u}_0\|_{\dH^2}^2.
\eeno
So that the condition \eqref{smalldelta} in Proposition \ref{stabH1} is fulfilled. Thus \eqref{eqtbaru} has a unique global strong solution $\bar{\vv u}$ such that
$$\|\bar{\vv u}-\wt{\vv u}\|_{L^\infty(\R^+;\dH^1)}^2
+\|\bar{\vv u}-\wt{\vv u}\|_{L^2(\R^+;\dH^2)}^2
\lesssim\min\bigl\{\|u^\th_0\|_{\dot{H}^1_{\rm axi}}^2,
\|{\vv u}_0\|_{\dH^1}^2\bigr\}
\exp\bigl(C\|{\vv u}_0\|_{L^2}^6\|{\vv u}_0\|_{\dH^2}^2\bigr),$$
which together with \eqref{estiwtudH1} leads to
the desired estimate \eqref{estibaruH1}.

The estimate \eqref{estibaruL2} follows from the Basic energy equality of Navier-Stokes equations. Thus, we complete the proof of the proposition.
\end{proof}

\subsection{The proof of Theorem \ref{thm1}}
The goal of this subsection is to prove Theorem \ref{thm1}.

\begin{proof}[The proof of Theorem \ref{thm1}] We divide the proof into two steps.

{\bf Step 1. Global solvability of \eqref{eqtu}. } The strategy is still to use Proposition \ref{stabH1}. To do this, we need to analyse
the difference
between $\vv u_0$ and $\bar{\vv u}_0$, which reads as follows:
$$\vv u_0-\bar{\vv u}_0=\bigl(u_0^r-\cM(u_0^r)\bigr)e_r
+\bigl(u_0^\th-\cM(u_0^\th)\bigr)e_\th
+\bigl(u_0^z-\cM(u_0^z)\bigr)e_z,$$
$$\wt\nabla(\vv u_0-\bar{\vv u}_0)
=\bigl(\wt\nabla u_0^r-\cM(\wt\nabla u_0^r)\bigr)e_r
+\bigl(\wt\nabla u_0^\th-\cM(\wt\nabla u_0^\th)\bigr)e_\th
+\bigl(\wt\nabla u_0^z-\cM(\wt\nabla u_0^z)\bigr)e_z,$$
and
$$\f{\pa_\th}r(\vv u_0-\bar{\vv u}_0)=\Bigl(\f{\pa_\th u_0^r}r
-\f{u_0^\th}r+\cM\bigl(\f{u_0^\th}r\bigr)\Bigr)e_r
+\Bigl(\f{\pa_\th u_0^\th}r
+\f{u_0^r}r-\cM\bigl(\f{u_0^r}r\bigr)\Bigr)e_\th
+\bigl(\f{\pa_\th u_0^z}r\bigr)e_z.$$
Then by using Lemma \ref{lemM} and the notation $\vv{\mathfrak{U}}_0=(\pa_\th u^r_0,\pa_\th u^\theta_0,
\pa_\th u^z_0),$ we can obtain
\begin{equation}\label{inequbaruinL2}
\|\vv u_0-\bar{\vv u}_0\|_{L^2}
\lesssim\|\vv{\mathfrak{U}}_0\|_{L^2},
\end{equation}
\begin{equation}\label{inequbaruinH1}
\|\nabla(\vv u_0-\bar{\vv u}_0)\|_{L^2}
\lesssim\|\wt\nabla\vv{\mathfrak{U}}_0\|_{L^2}
+\|r^{-1}\vv{\mathfrak{U}}_0\|_{L^2}=\|\vv{\mathfrak{U}}_0\|_{\dot{H}^1_{\rm axi}}.
\end{equation}

In view of the estimates \eqref{estibaruL2}-\eqref{inequbaruinH1}
and the smallness condition \eqref{smallThm1.2}, the condition \eqref{smalldelta} of Propostion \ref{stabH1} holds for $\vv u_0$ and $\vv v_0=\bar{\vv u}_0$.
Then by using  Proposition \ref{stabH1} again,
\eqref{eqtu} has a unique global strong solution $\vv u$ in
$C(\R^+;H^1)\cap L^2(\R^+;\dH^1\cap \dH^2)$, satisfying
\beq\label{errorubaru}\begin{aligned}
&\|\vv u-\bar{\vv u}\|_{L^\infty_t(L^2)}^2+\|\nabla(\vv u-\bar{\vv u})\|_{L^2_t(L^2)}^2
\lesssim\|\vv{\mathfrak{U}}_0\|_{L^2}^2
\exp\bigl(\exp(C\|{\vv u}_0\|_{L^2}^6\|{\vv u}_0\|_{\dH^2}^2)\bigr),\\
&\|\vv u-\bar{\vv u}\|_{L^\infty_t(\dH^1)}^2+\|\nabla(\vv u-\bar{\vv u})\|_{L^2_t(\dH^1)}^2
\lesssim\|\vv{\mathfrak{U}}_0\|_{\dot{H}^1_{\rm axi}}^2
\exp\bigl(\exp(C\|{\vv u}_0\|_{L^2}^6\|{\vv u}_0\|_{\dH^2}^2)\bigr).
\end{aligned}\eeq
In particular, the above estimates together with interpolation inequality and
the smallness condition \eqref{smallThm1.2} implies
$$\|\vv u-\bar{\vv u}\|_{L^\infty_t(\dH^{\f12})}^2
+\|\nabla(\vv u-\bar{\vv u})\|_{L^2_t(\dH^{\f12})}^2
\lesssim\|\vv{\mathfrak{U}}_0\|_{L^2}\|\vv{\mathfrak{U}}_0\|_{\dot{H}^1_{\rm axi}}
\exp\bigl(\exp(C\|{\vv u}_0\|_{L^2}^6\|{\vv u}_0\|_{\dH^2}^2)\bigr)
<\e,$$
which means that $\vv u$ is indeed close to an axisymmetric solution.

{\bf Step 2. Error estimates of $\cM\vv u-\bar{\vv u}$}

{\it Step 2.1. Equations for $\cM\vv u-\bar{\vv u}$.}
We first get, by applying $\cM$ to \eqref{eqtu} that
\begin{equation}\label{eqtMu}
\left\{
\begin{aligned}
& \pa_t\cM u^r+\cM\Bigl((u^r\pa_r+\ut\f{\pa_\th}r+u^z\pa_z)\ur\Bigr)\\
&\qquad\qquad\qquad\qquad
-\bigl(\D-\f1{r^2}\bigr)\cM u^r
-\cM\Bigl(\f{|u^\theta|^2}{r}\Bigr)+\pa_r\cM P=0,\\
& \pa_t\cM \ut+\cM\Bigl((u^r\pa_r+\ut\f{\pa_\th}r+u^z\pa_z)\ut\Bigr)
-\bigl(\D-\f1{r^2}\bigr)\cM u^\theta
+\cM\Bigl(\frac{u^r u^\theta}{r}\Bigr)=0,\\
& \pa_t\cM u^z+\cM\Bigl((u^r\pa_r+\ut\f{\pa_\th}r+u^z\pa_z)\uz\Bigr)
-\D\cM u^z+\pa_z\cM P=0,\\
& \pa_r\cM u^r+\frac{\cM\ur}r+\pa_z\cM u^z=0,\\
&\cM{\vv u}|_{t=0} =\cM{\vv u_0},
\end{aligned}
\right.
\end{equation}
where we have used the fact that for any regular enough function $f$, there holds
$$\int_0^{2\pi}\pa_\th f(r,\th,z)\,d\th=0.$$
Then by denoting $\vv v\eqdefa\cM \vv u-\bar{\vv u},
\,Q\eqdefa\cM P-\bar{P}$, and in view of \eqref{eqtbaru} and \eqref{eqtMu},
we deduce
\begin{equation}\label{Error NS 1}
\left\{
\begin{aligned}
&\p_t\vv v-\Delta\vv v+\na Q=\vv F+\vv G,\quad \dive\vv v=0,\\
&\vv v|_{t=0}=0,
\end{aligned}
\right.
\end{equation}
where $\vv F=F^r\vv e_r+F^\th\vv e_\th+F^z\vv e_z,~
\vv G=G^r\vv e_r+G^\th\vv e_\th+G^z\vv e_z$ with
\begin{align*}
&F^r=(\bar u^r\pa_r+\bar u^z\pa_z)\bar u^r
-\cM\Bigl((u^r\pa_r+u^z\pa_z)\ur\Bigr),~
G^r=\barut\bigl(\f{\pa_\th\barur}r-\f{\barut}{r}\bigr)
-\cM\Bigl(\ut\bigl(\f{\pa_\th\ur}r-\f{u^\theta}{r}\bigr)\Bigr),\\
& F^\th=(\bar u^r\pa_r+\bar u^z\pa_z)\bar u^\th
-\cM\Bigl((u^r\pa_r+u^z\pa_z)\ut\Bigr),
~G^\th=\barut\bigl(\f{\pa_\th\barut}r+\f{\barur}{r}\bigr)
-\cM\Bigl(\ut\bigl(\f{\pa_\th\ut}r+\f{u^r}{r}\bigr)\Bigr),\\
& F^z=(\barur\pa_r+\baruz\pa_z)\baruz
-\cM\Bigl((u^r\pa_r+u^z\pa_z)\uz\Bigr),~
G^z=-\cM\Bigl(\f{\ut\pa_\th\uz}r\Bigr).
\end{align*}
We mention that here $\pa_\th\barur,\,\pa_\th\barut$ indeed vanishes,
but we still write them to derive a symmetric form.
In this way, we can handle these terms by using the
following key observation: for any functions
$f,\,g$ and any axisymmetric functions $\bar f,\,\bar g$, there holds
\begin{equation}\begin{split}\label{id1}
\cM(fg)-\bar f\bar g&=\cM(fg-\bar f\bar g)\\
&=\cM\bigl((f-\bar f)(g-\bar g)+(f-\bar f)\bar g
+\bar f(g-\bar g)\bigr)\\
&=\cM\bigl((f-\bar f)(g-\bar g)\bigr)+\bar g\cM(f-\bar f)
+\bar f\cM(g-\bar g).
\end{split}\end{equation}
On the other hand, noting that $\bar f,\,\bar g$ are axisymmetric, we have
$$\pa_\th g=\pa_\th(g-\bar g),\andf\cM\bigl(\bar f\pa_\th(g-\bar g)\bigr)
=\bar f\cM\bigl(\pa_\th(g-\bar g)\bigr)=0.$$
As a result, there holds
\begin{equation}\label{id2}
\cM(f\pa_\th g)=\cM\bigl(f\pa_\th(g-\bar g)\bigr)
=\cM\bigl((f-\bar f)\pa_\th(g-\bar g)\bigr).
\end{equation}

For notation simplification, let us introduce $\vv V=\vv u-\bar{\vv u}$.
Obviously there holds
$$\cM\vv V=\cM\vv u-\cM\bar{\vv u}=\cM\vv u-\bar{\vv u}=\vv v.$$
Then by using \eqref{id1}, we have
\beno
\cM(u^r\p_ru^r)-\barur\p_r\barur=\cM(V^r\p_rV^r)+v^r\p_r\barur+\barur\p_rv^r.
\eeno
Similar formulas hold for the other terms in $F^r,\,F^\th,\,F^z$. And we can write
\beq\label{F 2}\begin{aligned}
&F^r=-\cM\Bigl((V^r\p_r+V^z\p_z)V^r\Bigr)-(v^r\p_r+v^z\p_z)\barur-(\barur\p_r+\baruz\p_z)v^r,\\
&F^\th=-\cM\Bigl((V^r\p_r+V^z\p_z)V^\th\Bigr)-(v^r\p_r+v^z\p_z)\barut-(\barur\p_r+\baruz\p_z)v^\th,\\
&F^z=-\cM\Bigl((V^r\p_r+V^z\p_z)V^z\Bigr)-(v^r\p_r+v^z\p_z)\baruz-(\barur\p_r+\baruz\p_z)v^z.
\end{aligned}\eeq

Exactly along the same line, by using \eqref{id1}, we have
\begin{align*}
G^r&=-\cM\Bigl(V^\th \bigl(\f{\p_\th V^r}{r}-\f{V^\th}{r}\bigr)\Bigr)
-\barut\cM\Bigl(\f{\p_\th V^r}{r}-\f{V^\th}{r}\Bigr)
-\bigl(\f{\p_\th\barur}{r}-\f{\barut}{r}\bigr)\cM(V^\th)\\
&=-\cM\Bigl(V^\th \bigl(\f{\p_\th V^r}{r}-\f{V^\th}{r}\bigr)\Bigr)
+2\f{\barut}r\vt,
\end{align*}
and
$$G^\th=-\cM\Bigl(V^\th \bigl(\f{\p_\th V^\th}{r}+\f{V^r}{r}\bigr)\Bigr)
-\f{\barut}r\vr-\f{\barur}r\vt.$$
Notice that
$$\f{\p_\th\vv V}{r}=(\f{\p_\th V^r}{r}-\f{V^\th}{r})\vv e_r
+(\f{\p_\th V^\th}{r}+\f{V^r}{r})\vv e_\th+\f{\p_\th V^z}{r}\vv e_z,$$
thus we can write
\begin{equation}\label{F 4}
G^r=-\cM\Bigl(V^\th\bigl(\f{\p_\th\vv V}{r}\cdot\vv e_r\bigr)\Bigr)
+2\f{\barut}r\vt,\andf
G^\th=-\cM\Bigl(V^\th\bigl(\f{\p_\th\vv V}{r}\cdot\vv e_\th\bigr)\Bigr)
-\f{\barut}r\vr-\f{\barur}r\vt.
\end{equation}

On the other hand, by using \eqref{id2}, we have
\begin{equation}\label{F 5}
G^z=-\cM\Bigl(V^\th\f{\p_\th V^z}{r}\Bigr)
=-\cM\Bigl(V^\th\bigl(\f{\p_\th\vv V}{r}\cdot\vv e_z\bigr)\Bigr).
\end{equation}

{\it Step 2.2. $L^2$-estimate for $\vv v$.} Taking $L^2$ inner product of \eqref{Error NS 1} with $\vv v$ gives rise to
\beq\label{F 6}
\f12\f{d}{dt}\|\vv v\|_{L^2}^2+\|\na\vv v\|_{L^2}^2=(\vv F\,|\,\vv v)+(\vv G\,|\,\vv v).
\eeq
Thanks to the formula \eqref{F 2} and Lemma \ref{lemM}, we obtain
\beno\begin{aligned}
|(\vv F\,|\,\vv v)|&\lesssim\|\vv V\|_{L^3}\|{\na}\vv V\|_{L^2}\|\vv v\|_{L^6}
+\|\vv v\|_{L^3}\|{\na}\bar{\vv u}\|_{L^2}\|\vv v\|_{L^6}
+\|\bar{\vv u}\|_{L^6}\|{\na}\vv v\|_{L^2}\|\vv v\|_{L^3}\\
&\lesssim\|\vv V\|_{L^2}^{\f12}\|\na\vv V\|_{L^2}^{\f32}\|\na\vv v\|_{L^2}+\|\na\bar{\vv u}\|_{L^2}\|\vv v\|_{L^2}^{\f12}\|\na\vv v\|_{L^2}^{\f32},
\end{aligned}\eeno
which along with Young's inequality implies
\beq\label{F 7}
|(\vv F\,|\,\vv v)|\leq\f14\|\na\vv v\|_{L^2}^2+C\|\vv V\|_{L^2}\|\na\vv V\|_{L^2}^3+\|\na\bar{\vv u}\|_{L^2}^4\|\vv v\|_{L^2}^2.
\eeq
While by virtue of \eqref{F 4},~\eqref{F 5}, Lemma \ref{lemM}
and Lemma \ref{lembaru}, we get
\begin{align*}
|(\vv G\,|\,\vv v)|&\lesssim\|V^\th\|_{L^3}
\bigl\|\f{\p_\th\vv V}{r}\bigr\|_{L^2}\|\vv v\|_{L^6}
+\|\vv v\|_{L^3}\bigl(\bigl\|\f{\barur}{r}\bigr\|_{L^2}
+\bigl\|\f{\barut}{r}\bigr\|_{L^2}\bigr)\|\vv v\|_{L^6}\\
&\lesssim\|\vv V\|_{L^2}^{\f12}\|\na\vv V\|_{L^2}^{\f32}\|\na\vv v\|_{L^2}
+\|\vv v\|_{L^2}^{\f12}\|\na\vv v\|_{L^2}^{\f32}\|\na\bar{\vv u}\|_{L^2}.
\end{align*}
Then by using Young's inequality, we deduce
\beq\label{F 8}
|(\vv G\,|\,\vv v)|\leq\f14\|\na\vv v\|_{L^2}^2
+C\|\vv V\|_{L^2}\|\na\vv V\|_{L^2}^3
+C\|\vv v\|_{L^2}^2\|\na\bar{\vv u}\|_{L^2}^4.
\eeq

Now by substituting \eqref{F 7} and \eqref{F 8}
into \eqref{F 6}, we obtain
\beno
\f{d}{dt}\|\vv v\|_{L^2}^2+\|\na\vv v\|_{L^2}^2\leq C\|\vv V\|_{L^2}\|\na\vv V\|_{L^2}^3+C\|\vv v\|_{L^2}^2\|\na\bar{\vv u}\|_{L^2}^4.
\eeno
Then by using Gronwall's inequality and the fact that $\vv v|_{t=0}=0$, we achieve
$$\|\vv v\|_{L^\infty_t(L^2)}^2+\|\na\vv v\|_{L^2_t(L^2)}^2
\lesssim\|\vv V\|_{L^\infty_t(L^2)}\|\na\vv V\|_{L^\infty_t(L^2)}\|\na\vv V\|_{L^2_t(L^2)}^2
\exp\bigl(C\|\na\bar{\vv u}\|_{L^\infty_t(L^2)}^2\|\na\bar{\vv u}\|_{L^2_t(L^2)}^2\bigr).$$
Recall that $\vv V=\vv u-\bar{\vv u}$,
by using \eqref{estibaruL2},~\eqref{estibaruH1} and \eqref{errorubaru}, we obtain
\begin{equation}\label{F 9}
\|\vv v\|_{L^\infty_t(L^2)}^2+\|\vv v\|_{L^2_t(L^2)}^2
\lesssim\|\vv{\mathfrak{U}}_0\|_{L^2}^3
\|\vv{\mathfrak{U}}_0\|_{\dot{H}^1_{\rm axi}}
\exp\bigl(\exp(C\|{\vv u}_0\|_{L^2}^6\|{\vv u}_0\|_{\dH^2}^2)\bigr).
\end{equation}
This together with the smallness assumption \eqref{smallThm1.2} gives
$$\|\vv v\|_{L^\infty_t(L^2)}^2+\|\vv v\|_{L^2_t(L^2)}^2
\lesssim\e\|\vv{\mathfrak{U}}_0\|_{L^2}^2
\exp\bigl(\exp(C\|{\vv u}_0\|_{L^2}^6\|{\vv u}_0\|_{\dH^2}^2)\bigr),$$
which is exactly the first inequality of \eqref{error estimate 2}.

{\it Step 2.3. $\dH^1$-estimate for $\vv v$.} Taking $L^2$ inner product of \eqref{Error NS 1} with $-\D\vv v$ gives rise to
\begin{equation}\begin{split}\label{F 10}
\f12\f{d}{dt}\|\na\vv v\|_{L^2}^2+\|\D\vv v\|_{L^2}^2
&=-(\vv F\,|\,\D\vv v)-(\vv G\,|\,\D\vv v)\\
&\leq\bigl(\|\vv F\|_{L^2}+\|\vv G\|_{L^2}\bigr)\|\D\vv v\|_{L^2}.
\end{split}\end{equation}
In view of the formulas \eqref{F 2}-\eqref{F 5},
we can use Lemmas \ref{lemM} and \ref{lembaru} to get
\begin{align*}
\|\vv F\|_{L^2}+\|\vv G\|_{L^2}
\lesssim&\|\vv V\|_{L^\infty}\bigl(\|\wt{\na}\vv V\|_{L^2}
+\bigl\|\f{\p_\th\vv V}{r}\bigr\|_{L^2}\bigr)\\
&+\|\vv v\|_{L^\infty}\bigl(\|\wt{\na}\bar{\vv u}\|_{L^2}
+\bigl\|\f{\barur}{r}\bigr\|_{L^2}+\bigl\|\f{\barut}{r}\bigr\|_{L^2}\bigr)
+\|\bar{\vv u}\|_{L^\infty}\|\wt{\na}\vv v\|_{L^2}\\
\lesssim&\|\vv V\|_{L^\infty}\|\na\vv V\|_{L^2}+\|\vv v\|_{L^\infty}\|\na\bar{\vv u}\|_{L^2}+\|\bar{\vv u}\|_{L^\infty}\|\na\vv v\|_{L^2}\\
\lesssim&\|\na\vv V\|_{L^2}^{\f32}\|\D\vv V\|_{L^2}^{\f12}
+\|\na\vv v\|_{L^2}^{\f12}\|\D\vv v\|_{L^2}^{\f12}\|\na\bar{\vv u}\|_{L^2}
+\|\na\bar{\vv u}\|_{L^2}^{\f12}\|\D\bar{\vv u}\|_{L^2}^{\f12}\|\na\vv v\|_{L^2},
\end{align*}
which together with Young's inequality leads to
\begin{align*}
\bigl(\|\vv F\|_{L^2}+&\|\vv G\|_{L^2}\bigr)\|\D\vv v\|_{L^2}\\
&\leq\f12\|\D\vv v\|_{L^2}^2+C\|\na\vv V\|_{L^2}^3\|\D\vv V\|_{L^2}
+C\bigl(\|\na\bar{\vv u}\|_{L^2}^4
+\|\na\bar{\vv u}\|_{L^2}\|\D\bar{\vv u}\|_{L^2}\bigr)\|\na\vv v\|_{L^2}^2.
\end{align*}
By substituting this into \eqref{F 10}, we infer
\begin{align*}
\f{d}{dt}\|\na\vv v\|_{L^2}^2+\|\D\vv v\|_{L^2}^2
\lesssim\|\na\vv V\|_{L^2}^3\|\D\vv V\|_{L^2}
+\bigl(\|\na\bar{\vv u}\|_{L^2}^4
+\|\na\bar{\vv u}\|_{L^2}\|\D\bar{\vv u}\|_{L^2}\bigr)\|\na\vv v\|_{L^2}^2,
\end{align*}
Then by applying Gronwall's inequality and using the fact that $\vv v|_{t=0}=0$,
we deduce
\begin{align*}
\|\na\vv v\|_{L^\infty_t(L^2)}^2+\|\D\vv v\|_{L^2_t(L^2)}^2
\leq& C\|\na\vv V\|_{L^\infty_t(L^2)}^2\|\na\vv V\|_{L^2_t(L^2)}\|\D\vv V\|_{L^2_t(L^2)}\\
&\times\exp\bigl(C\|\na\bar{\vv u}\|_{L^\infty_t(L^2)}^2\|\na\bar{\vv u}\|_{L^2_t(L^2)}^2
+C\|\na\bar{\vv u}\|_{L^2_t(L^2)}\|\D\bar{\vv u}\|_{L^2_t(L^2)}\bigr).
\end{align*}
Notice that $\vv V=\vv u-\bar{\vv u}$, then we can use
\eqref{estibaruL2},~\eqref{estibaruH1} and \eqref{errorubaru} to obtain
\beq\label{F 12}
\|\na\vv v\|_{L^\infty_t(L^2)}^2+\|\D\vv v\|_{L^2_t(L^2)}^2
\lesssim\|\vv{\mathfrak{U}}_0\|_{L^2}
\|\vv{\mathfrak{U}}_0\|_{\dot{H}^1_{\rm axi}}^3
\exp\bigl(\exp(C\|{\vv u}_0\|_{L^2}^6\|{\vv u}_0\|_{\dH^2}^2)\bigr).
\eeq
This together with the smallness assumption \eqref{smallThm1.2} gives
$$\|\nabla\vv v\|_{L^\infty_t(L^2)}^2+\|\D\vv v\|_{L^2_t(L^2)}^2
\lesssim\e\|\vv{\mathfrak{U}}_0\|_{\dot{H}^1_{\rm axi}}^2
\exp\bigl(\exp(C\|{\vv u}_0\|_{L^2}^6\|{\vv u}_0\|_{\dH^2}^2)\bigr),$$
which is exactly the second inequality of \eqref{error estimate 2}.

Till now, we have completed the proof of Theorem \ref{thm1}.
\end{proof}

\setcounter{equation}{0}
\section{The proof of Theorem \ref{thm2}}
As we have mentioned in Subsection \ref{subsectionmainthm},
for any initial data given by \eqref{initialFourier},
Theorem \ref{thm1} guarantees the existence of some positive constant $\ve_0$ such that
for any $\ve\in(0,\ve_0)$, \eqref{NS} has a unique global solution
$\vv u\in C(\R_+;H^1)\cap L^2(\R_+;\dH^1\cap \dH^2)$.

By virtue of the classical result that any strong solution
to N-S will become analytic at positive time, hence we can expand
this solution into Fourier series in $\th$ variable as
\begin{equation}\label{appve2}
\left\{
\begin{split}
\ur(t,x)=&\bigl(\ur_{(0),0}+\ve\ur_{(1),0}\bigr)(t,r,z)
+\sum_{k=1}^\infty\bigl(\ur_{(0),k}+\ve\ur_{(1),k}\bigr)(t,r,z)\cos k\th\\
&+\sum_{k=1}^\infty\bigl(\vr_{(0),k}+\ve\vr_{(1),k}\bigr)(t,r,z)\sin k\th
+\cO(\ve^2),\\
\ut(t,x)=&\bigl(\ut_{(0),0}+\ve\ut_{(1),0}\bigr)(t,r,z)
+\sum_{k=1}^\infty\bigl(\ut_{(0),k}+\ve\ut_{(1),k}\bigr)(t,r,z)\cos k\th\\
&+\sum_{k=1}^\infty\bigl(\vt_{(0),k}+\ve\vt_{(1),k}\bigr)(t,r,z)\sin k\th
+\cO(\ve^2),\\
\uz(t,x)=&\bigl(\uz_{(0),0}+\ve\uz_{(1),0}\bigr)(t,r,z)
+\sum_{k=1}^\infty\bigl(\uz_{(0),k}+\ve\uz_{(1),k}\bigr)(t,r,z)\cos k\th\\
&+\sum_{k=1}^\infty\bigl(\vz_{(0),k}+\ve\vz_{(1),k}\bigr)(t,r,z)\sin k\th
+\cO(\ve^2).
\end{split}
\right.
\end{equation}
Correspondingly, we can expand the pressure $P=(-\D)^{-1}\dive\dive(\vv u\otimes\vv u)$
into
\begin{equation}\begin{split}
P(t,x)=&\bigl(P_{(0),0}+\ve P_{(1),0}\bigr)(t,r,z)
+\sum_{k=1}^\infty\bigl(P_{(0),k}+\ve P_{(1),k}\bigr)(t,r,z)\cos k\th\\
&+\sum_{k=1}^\infty\bigl(Q_{(0),k}+\ve Q_{(1),k}\bigr)(t,r,z)\sin k\th+\cO(\ve^2).
\end{split}\end{equation}

In the following, we shall give the explicit formulas
for the profiles in the expansion \eqref{appve2},
and verify the validity of this expansion
in $L^\infty(\R^+;L^2\cap L^\infty)$ sense.
For notation simplification, we shall use $\vv u_{(i)}$
to denote the $\ve^i$ order term in the expansion of $\vv u$, i.e.
\begin{equation}
\vv u=\vv u_{(0)}+\ve\vv u_{(1)}+\cO(\ve^2),\quad
\vv u_{(i)}=\vv u_{(i),0}+\sum_{k=1}^\infty\Bigl(\vv u_{(i),k}\cos k\th
+\vv v_{(i),k}\sin k\th\Bigr),
\end{equation}
and $\vv u_{(i)}$ does not rely on $\ve$.

\subsection{The $\ve^0$ order terms}
We make the following {\bf Ansatz $1$:} $\vv u_{(0)}=\ur_{(0),0}\vv e_r+\uz_{(0),0}\vv e_z$,
where $\ur_{(0),0}$ and $\uz_{(0),0}$ satisfy the following axisymmetric without swirl N-S:
\begin{equation}\label{deteu0}
\left\{
\begin{split}
& \pa_t\ur_{(0),0}+(\ur_{(0),0}\pa_r+\uz_{(0),0}\pa_z)\ur_{(0),0}
-(\pa_r^2+\pa_z^2+\frac 1r\pa_r-\frac{1}{r^2})\ur_{(0),0}+\pa_r P_{(0),0}=0,\\
& \pa_t\uz_{(0),0}+(\ur_{(0),0}\pa_r+\uz_{(0),0}\pa_z)u^z_{(0),0}
-(\pa_r^2+\pa_z^2+\frac 1r\pa_r)\uz_{(0),0}+\pa_z P_{(0),0}=0,\\
& \pa_r\ur_{(0),0}+\frac 1r\ur_{(0),0}+\pa_z\uz_{(0),0}=0,\\
& \ur_{(0),0}|_{t=0}=a^r_{0},\quad
\uz_{(0),0}|_{t=0}=a^z_{0}.
\end{split}
\right.
\end{equation}

To verify this ansatz, we first deduce from
Proposition \ref{P 1} that
\begin{equation}\label{boundu0}
\|\vv u_{(0)}\|_{L^\infty(\R^+; H^1)}^2+\|\nabla\vv u_{(0)}\|_{L^2(\R^+; H^1)}^2
\leq C_0.
\end{equation}
Here and in all that follows, $C_{0}$ denotes some positive constant
depending only on the norms of the profiles of $\vv u(0,x)$
appearing in \eqref{1.14}, and may be different in each appearance.

On the other hand, in view of the expression
\eqref{initialFourier} for initial data, we have
\beno
\vv u|_{t=0}-\vv u_{(0)}|_{t=0}=\ve a^\th_{0}(r,z)\vv e_\th+\ve\sum_{k=1}^\infty
\Bigl(\vv a_{k}(r,z)\cos(k\th)+\vv b_k(r,z)\sin(k\th)\Bigr).
\eeno
By using Parseval's identities, there holds
\beq\label{F 14}
\bigl\|\vv u|_{t=0}-\vv u_{(0)}|_{t=0}\bigr\|_{H^1}\leq C_0\ve.
\eeq
Then for sufficiently small $\ve>0$, by virtue of \eqref{boundu0}, \eqref{F 14} and
Proposition \ref{stabH1}, we obtain
\begin{equation}\label{4.44}
\|\vv u-\vv u_{(0)}\|_{L^\infty(\R^+;H^1\cap L^\infty)}
+\|\nabla(\vv u-\vv u_{(0)})\|_{L^2(\R^+;H^1)}\leq C_{0}\ve.
\end{equation}

Notice that the system \eqref{deteu0}
does not rely on $\ve$, and thus its solution $(\ur_{(0),0},\uz_{(0),0})$
also does not rely on $\ve$. This together with the
estimate \eqref{4.44} guarantees the validity of our Ansatz $1$,
namely $\ur_{(0),0}\vv e_r+\uz_{(0),0}\vv e_z$ is indeed the
$\ve^0$ order term in the expansion of $\vv u$.

In particular, we have shown that $\ut_{(0),0}$ vanishes,
and for every $k\in\N$, the $k$-th Fourier coefficient does not contain
$\ve^0$ order terms, just the same as the initial data \eqref{initialFourier}.

\subsection{Derivation of the $\ve^1$ order terms
in the Euclidean coordinates}\label{sub4.2}
Let us temporarily go back to the Euclidean coordinates,
and consider $\vv u_{(1)}$ as a whole part.

We make the following
{\bf Ansatz $2$:}
$\vv u_{(1)}$ is a solution
to the linearization of the perturbed Navier-Stokes system around
$\vv u_{(0)}$, precisely
\begin{equation}\label{eqtuve1}
\left\{
\begin{split}
& \pa_t\vv u_{(1)}+\vv u_{(0)}\cdot\nabla\vv u_{(1)}
+\vv u_{(1)}\cdot\nabla\vv u_{(0)}
-\D\vv u_{(1)}=-\nabla P_{(1)},\\
& \dive\vv u_{(1)}=0,\\
&\vv u_{(1)}|_{t=0}=a^\th_0\vv e_\th+\sum_{k=1}^\infty\Bigl(\vv a_{k}\cos k\th
+\vv b_{k}\sin k\th\Bigr).
\end{split}
\right.
\end{equation}

To verify this ansatz, we first get, by taking $L^2$ inner product of \eqref{eqtuve1}
with $\vv u_{(1)}-\D\vv u_{(1)}$ that
\begin{align*}
\f12\f{d}{dt}&\|\vv u_{(1)}\|_{H^1}^2+\|\nabla\vv u_{(1)}\|_{H^1}^2
\leq\bigl|\bigl(\vv u_{(0)}\cdot\nabla\vv u_{(1)}
+\vv u_{(1)}\cdot\nabla\vv u_{(0)})
\,\big|\,\vv u_{(1)}-\D\vv u_{(1)}\bigr)\bigr|\\
&\leq\|\vv u_{(1)}\|_{L^3}\|\nabla\vv u_{(0)}\|_{L^2}
\|\vv u_{(1)}\|_{L^6}+\bigl(\|\vv u_{(0)}\|_{L^6}\|\nabla\vv u_{(1)}\|_{L^3}
+\|\vv u_{(1)}\|_{L^\infty}\|\nabla\vv u_{(0)}\|_{L^2}\bigr)
\|\D\vv u_{(1)}\|_{L^2}\\
&\leq\f12\|\D\vv u_{(1)}\|_{L^2}^2+C\|\vv u_{(1)}\|_{H^1}^2\|\nabla\vv u_{(0)}\|_{L^2}^4.
\end{align*}
Subtracting $\f12\|\D\vv u_{(1)}\|_{L^2}^2$ on both sides,
and then applying Gronwall's inequality gives
\begin{equation}\label{boundu11}
\|\vv u_{(1)}\|_{L^\infty_t(H^1)}^2+\|\nabla\vv u_{(1)}\|_{L^2_t(H^1)}^2
\leq\|\vv u_{(1)}|_{t=0}\|_{H^1}^2
\exp\bigl(C\|\nabla\vv u_{(0)}\|_{L^\infty_t(L^2)}^2
\|\nabla\vv u_{(0)}\|_{L^2_t(L^2)}^2\bigr)
\leq C_0.
\end{equation}

While following similar derivation as \eqref{estidomL4}, we deduce
\begin{equation}\label{boundu1}
\|\na\vv u_{(1)}\|_{L^\infty_t(L^4)}\lesssim\|\curl\vv u_{(1)}\|_{L^\infty_t(L^4)}\leq C_0.
\end{equation}

On the other hand, let us consider $\vv R_{(1)}
\eqdef\vv u-\vv u_{(0)}-\ve\vv u_{(1)}$, which satisfies
\begin{equation}\label{eqtdeltave1}
\left\{
\begin{split}
& \pa_t\vv R_{(1)}+\vv u\cdot\nabla\vv R_{(1)}
+\vv R_{(1)}\cdot\nabla(\vv u_{(0)}+\ve\vv u_{(1)})
+\ve^2\vv u_{(1)}\cdot\nabla\vv u_{(1)}-\D\vv R_{(1)}
=-\nabla\Pi_{(1)},\\
& \dive\vv R_{(1)}=0,\\
&\vv R_{(1)}|_{t=0}=0,
\end{split}
\right.
\end{equation}
where $\Pi_{(1)}$ can be obtained by taking divergence operator
to the first equation in \eqref{eqtdeltave1}.

In the following, we shall derive $L^2,~\dH^1$ and $\dot{W}^{1,4}$
estimate for $\vv R_{(1)}$.
Notice that if there is no external force term
$\ve^2\vv u_{(1)}\cdot\nabla\vv u_{(1)}$ in \eqref{eqtdeltave1},
then this type of system has already been studied in Proposition \ref{stabH1}.
So here we only focus on the estimate for this external force term.

Firstly, we have
\begin{align*}
\Bigl|\int_{\R^3}\ve^2\bigl(\vv u_{(1)}\cdot\nabla\vv u_{(1)}\bigr)
\vv R_{(1)}\,dx\Bigr|
&\leq C\ve^2\|\vv u_{(1)}\|_{L^3}
\|\nabla\vv u_{(1)}\|_{L^2}\|\vv R_{(1)}\|_{L^6}\\
&\leq \f12\|\nabla\vv R_{(1)}\|_{L^2}^2
+C\ve^4\|\nabla\vv u_{(1)}\|_{L^2}^3\|\vv u_{(1)}\|_{L^2},
\end{align*}
and in view of the bound \eqref{boundu11}, there holds
$$\ve^4\int_0^t\|\nabla\vv u_{(1)}\|_{L^2}^3\|\vv u_{(1)}\|_{L^2}\,dt'
\leq\ve^4\|\nabla\vv u_{(1)}\|_{L^2_t(L^2)}^2
\|\nabla\vv u_{(1)}\|_{L^\infty_t(L^2)}\|\vv u_{(1)}\|_{L^\infty_t(L^2)}
\leq C_0\ve^4.$$
Then a similar procedure as the proof of Proposition \ref{stabH1} leads to
\begin{equation}\label{bounddelta1L2}
\|\vv R_{(1)}\|_{L^\infty_t(L^2)}^2+\|\nabla\vv R_{(1)}\|_{L^2_t(L^2)}^2
\leq C_0\ve^4.
\end{equation}

Similarly, we have
\begin{align*}
\Bigl|\int_{\R^3}\ve^2\bigl(\vv u_{(1)}\cdot\nabla\vv u_{(1)}\bigr)
\cdot\D\vv R_{(1)}\,dx\Bigr|
&\leq\ve^2\|\vv u_{(1)}\|_{L^\infty}\|\nabla\vv u_{(1)}\|_{L^2}
\|\D\vv R_{(1)}\|_{L^2}\\
&\leq \f12\|\D\vv R_{(1)}\|_{L^2}^2
+\ve^4\|\nabla\vv u_{(1)}\|_{L^2}^3\|\nabla\vv u_{(1)}\|_{\dH^1},
\end{align*}
and
\begin{align*}
&\Bigl|\int_{\R^3}\ve^2\curl\bigl(\vv u_{(1)}\cdot\nabla\vv u_{(1)}\bigr)
\cdot|\curl\vv R_{(1)}|^2\curl\vv R_{(1)}\,dx\Bigr|\\
&\leq\ve^2\|\vv u_{(1)}\|_{L^6}\|\nabla\vv u_{(1)}\|_{L^4}
\|\curl\vv R_{(1)}\|_{L^{12}}
\bigl(\bigl\|\nabla|\curl\vv R_{(1)}|^2\bigr\|_{L^2}
+\bigl\||\curl\vv R_{(1)}|\cdot\nabla \curl\vv R_{(1)}\bigr\|_{L^2}\bigr)\\
&\leq \f14\bigl(\bigl\|\nabla|\curl\vv R_{(1)}|^2\bigr\|_{L^2}^2
+\bigl\||\curl\vv R_{(1)}|\cdot\nabla \curl\vv R_{(1)}\bigr\|_{L^2}^2\bigr)
+\ve^8\|\na\vv u_{(1)}\|_{L^2}^4\|\na\vv u_{(1)}\|_{L^4}^4.
\end{align*}
In view of the bounds \eqref{boundu11} and \eqref{boundu1}, there holds
$$\ve^4\int_0^t\|\nabla\vv u_{(1)}\|_{L^2}^3\|\nabla\vv u_{(1)}\|_{\dH^1}\,dt'
\leq C_0\ve^4,\andf
\ve^8\int_0^t\|\na\vv u_{(1)}\|_{L^2}^4\|\na\vv u_{(1)}\|_{L^4}^4\,dt'
\leq C_0\ve^8.$$
Then a similar procedure as the proof of Proposition \ref{stabH1} leads to
\begin{equation}\label{boundnabladelta1L2L4}
\|\vv R_{(1)}\|_{L^\infty_t(\dH^1)}^2
+\|\nabla\vv R_{(1)}\|_{L^2_t(\dH^1)}^2
\leq C_0\ve^4,\andf
\|\nabla\vv R_{(1)}\|_{L^\infty_t(L^4)}^4\leq C_0\ve^8.
\end{equation}

Now in view of the estimates \eqref{bounddelta1L2}
and \eqref{boundnabladelta1L2L4},
together with the interpolation inequality
$\|f\|_{L^\infty}\lesssim\|\nabla f\|_{L^2}^{\f13}
\|\nabla f\|_{L^4}^{\f23}$, we achieve
\begin{equation}\label{bounddelta1H1}
\|\vv R_{(1)}\|_{L^\infty_t(H^1\cap L^\infty)}^2
+\|\nabla\vv R_{(1)}\|_{L^2_t(H^1)}^2
\leq C_0\ve^4.
\end{equation}

In view of the estimate \eqref{bounddelta1H1} for the remainder
$\vv R_{(1)}$,
and the fact that $\vv u_{(1)}$ determined by \eqref{eqtuve1}
does not rely on $\ve$, we have verified our Ansatz $2$.

\begin{rmk}\label{rmk4.1}
Exactly along the same line, we can prove by induction that
for any $n\geq2$, $\vv u_{(n)}$ is a solution
to the following linearized system:
\begin{equation*}
\left\{
\begin{split}
& \pa_t\vv u_{(n)}+\vv u_{(0)}\cdot\nabla\vv u_{(n)}
+\vv u_{(n)}\cdot\nabla\vv u_{(0)}
-\D\vv u_{(n)}+\nabla P_{(n)}=-\sum_{i=1}^{n-1}
\vv u_{(i)}\cdot\nabla\vv u_{(n-i)},\\
& \dive\vv u_{(n)}=0,\\
&\vv u_{(n)}|_{t=0}=0.
\end{split}
\right.
\end{equation*}
And the remainder term $\vv R_{(n)}
\eqdef\vv u-\vv u_{(0)}-\ve\vv u_{(1)}-\cdots-\ve^n\vv u_{(n)}$ satisfies
$$\|\vv R_{(n)}\|_{L^\infty_t(H^1\cap L^\infty)}
\leq C_0\ve^{n+1}.$$
\end{rmk}

\subsection{Verification of $u^r_{(1),0}=u^z_{(1),0}=0$}
Noticing that $\vv u_{(0)}=\ur_{(0),0}\vv e_r+\uz_{(0),0}\vv e_z$
is axisymmetric without swirl, it is not difficult to deduce from \eqref{eqtuve1}
that $\bigl(u^r_{(1),0},u^z_{(1),0}\bigr)$ satisfies
\beq\label{eq for ur uz 1 0}
\left\{
\begin{aligned}
&\p_t\ur_{(1),0}+\vv u_{(0)}\cdot\na\ur_{(1),0}
+\wt{\vv u}_{(1),0}\cdot\wt\na\ur_{(0),0}-(\D-\f{1}{r^2})\ur_{(1),0}+\p_r P_{(1),0}=0,\\
&\p_t\uz_{(1),0}+\vv u_{(0)}\cdot\na\uz_{(1),0}+\wt{\vv u}_{(1),0}\cdot\wt\na\uz_{(0),0}-\D\uz_{(1),0}+\p_z P_{(1),0}=0,\\
&\pa_r\ur_{(1),0}+\frac 1r\ur_{(1),0}+\pa_z\uz_{(1),0}=0,\\
&\ur_{(1),0}|_{t=0} =\uz_{(1),0}|_{t=0}=0,
\end{aligned}
\right.
\eeq
where $\wt{\vv u}_{(1),0}\eqdef\ur_{(1),0}\vv e_r+\uz_{(1),0}\vv e_z$,
and the divergence-free condition here follows from
\beno\begin{aligned}
\dive\vv u_{(1),0}=&\p_ru^r_{(1),0}+\f{u^r_{(1),0}}{r}+\p_zu^z_{(1),0}\\
=&\cM\Bigl(\p_r\ur_{(1)}+\f{\ur_{(1)}}{r}
+\f{\p_\th\ut_{(1)}}{r}+\pa_z\uz_{(1)}\Bigr)=\cM(\dive\vv u_{(1)})=0.
\end{aligned}\eeno

By taking $L^2$ inner product of \eqref{eq for ur uz 1 0} with $\bigl(u^r_{(1),0},u^z_{(1),0}\bigr)$, we get
\beq\label{H 1}\f12\f{d}{dt}\|\wt{\vv u}_{(1),0}\|_{L^2}^2
+\|\na\wt{\vv u}_{(1),0}\|_{L^2}^2
=-\bigl(\wt{\vv u}_{(1),0}\cdot\wt\na\ur_{(0),0}\,|\,u^r_{(1),0}\bigr)
-\bigl(\wt{\vv u}_{(1),0}\cdot\wt\na\uz_{(0),0}\,|\,u^z_{(1),0}\bigr),
\eeq
where we used the fact that $\f{\p_\th}{r}P_{(1),0}(t,r,z)=0$ and then the following equality
\beno
-\bigl(\wt{\na} P_{(1),0}\,|\,\wt{\vv u}_{(1),0}\bigr)
=-\bigl(\na P_{(1),0}\,|\,\vv u_{(1),0}\bigr)=\bigl(P_{(1),0}\,|\,\dive\vv u_{(1),0}\bigr)=0.
\eeno

For the terms on the right-hand side of \eqref{H 1}, we have
\beno\begin{aligned}
\bigl|\bigl(\wt{\vv u}_{(1),0}\cdot\na\ur_{(0),0}\,|\,u^r_{(1),0}\bigr)\bigr|
+\bigl|\bigl(\wt{\vv u}_{(1),0}\cdot&\na\uz_{(0),0}\,|\,u^z_{(1),0}\bigr)\bigr|
\leq C\|\na\vv u_{(0)}\|_{L^2}\|\wt{\vv u}_{(1),0}\|_{L^4}^2\\
&\leq C\|\na\vv u_{(0)}\|_{L^2}
\|\wt{\vv u}_{(1),0}\|_{L^2}^{\f12}\|\na\wt{\vv u}_{(1),0}\|_{L^2}^{\f32}\\
&\leq\f12\|\na\wt{\vv u}_{(1),0}\|_{L^2}^2
+C\|\wt{\vv u}_{(1),0}\|_{L^2}^2\|\na\vv u_{(0)}\|_{L^2}^4.
\end{aligned}\eeno
As a result, we deduce
$$\f{d}{dt}\|\wt{\vv u}_{(1),0}\|_{L^2}^2
+\|\na\wt{\vv u}_{(1),0}\|_{L^2}^2
\leq C\|\wt{\vv u}_{(1),0}\|_{L^2}^2\|\na\vv u_{(0)}\|_{L^2}^4.$$
Then by applying Gronwall's inequality, together with the fact that
initially $\wt{\vv u}_{(1),0}|_{t=0}=0$, we deduce that
$\wt{\vv u}_{(1),0}$ actually vanishes for all time.
This completes the proof of Theorem \ref{thm2}.

\section{The proof of Theorem \ref{thm3}}\label{sec5}

As the initial data \eqref{inF} considered here is a special case of \eqref{initialFourier},
Theorem \ref{thm1} guarantees the existence of a unique global solution
$\vv u\in C(\R_+;H^1)\cap L^2(\R_+;\dH^1\cap \dH^2)$ for sufficiently small $\ve$.
And this solution can be expanded as \eqref{1.18}.
Hence the aim in the following
is to show that $\ur$ and $\uz$ remain even in $\th$,
while $\ut$ keeps odd in $\th$ for all time, namely
$$\vv u_{(1),0}=0,\quad\ut_{(j),0}=0,\quad\vr_{(j),k}=\ut_{(j),k}=\vz_{(j),k}=0,\quad\forall\
j\in\N,~ k\in\N.$$

\subsection{Verification of $\vv u_{(1),0}=\vv 0$}
For the initial data $\vv u(0,x)$ given by \eqref{inF}, we have
\beno
\cM(\vv u|_{t=0})=a^r_{0}(r,z)\vv e_r+a^z_{0}(r,z)\vv e_z.
\eeno
Thus the quantity $\bar{\vv u}$ defined by \eqref{eqtbaru},
is exactly $\ur_{(0),0}\vv e_r+\uz_{(0),0}\vv e_z$,
which sloves the axisymmetric without swirl N-S \eqref {deteu0} for the case here.

On the other hand, noticing that $\ut|_{t=0}\thicksim\ve$ and
$$\pa_\th\ur|_{t=0}=-\ve\sum_{k=1}^\infty
k a^r_{k}\sin k\th,\quad
\pa_\th\ut|_{t=0}=\ve\sum_{k=1}^\infty
k b^\th_{k}\cos k\th,\quad
\pa_\th\uz|_{t=0}=-\ve\sum_{k=1}^\infty
k a^z_{k}\sin k\th,$$
we can get by using the refined estimate \eqref{error estimate 2}
in Theorem \ref{thm1} that
$$\|\cM(\vv u)-\ur_{(0),0}\vv e_r-\uz_{(0),0}\vv e_z
\|_{L^\infty(\R^+;L^2)}\leq C_0\ve^2.$$
This together with the fact that
$\cM(\vv u)=\vv u_{(0),0}+\ve\vv u_{(1),0}+\ve^2\vv u_{(2),0}+\cdots$ implies
\beq\label{S 3}
\vv u_{(0),0}=\ur_{(0),0}\vv e_r+\uz_{(0),0}\vv e_z,\andf
\vv u_{(1),0}=0.
\eeq

\subsection{Verification of $\vr_{(1),k}=\ut_{(1),k}=\vz_{(1),k}=0$}\label{sub5.2}
In view of \eqref{S 3}, we can write
\begin{equation}\label{5.2}
\vv u_{(1)}=\sum_{k=1}^\infty
\Bigl(\vv u_{(1),k}\cos k\th+\vv v_{(1),k}\sin k\th\Bigr),
\end{equation}
and the corresponding pressure
$$P_{(1)}=\sum_{k=1}^\infty
\Bigl(P_{(1),k}\cos k\th+Q_{(1),k}\sin k\th\Bigr).$$
And from the proof of Subsection \ref{sub4.2}, we know that
$(\vv u_{(1)},P_{(1)})$ satisfies \eqref{eqtuve1}.

Then it is crucial to notice that $\vv u_{(0)}$ is axisymmetric without swirl,
which makes the couplings between each component of $\vv u_{(1)}$
in $(4.8)$ not so strong.
In fact, by rewriting \eqref{eqtuve1} according to
the basis $\{\cos k\th,\sin k\th:k\in\N\}$, it is not difficult to find that
we can decompose $\vv u_{(1)}$ and $P_{(1)}$ into two parts:
$\vv u_{(1)}^\prime+\vv u_{(1)}^{\prime\prime}$ and
$P_{(1)}^{\prime}+P_{(1)}^{\prime\prime}$ respectively, where
$$\vv u_{(1)}^{\prime}\eqdef\sum_{k=1}^\infty\Bigl(\ur_{(1),k}\cos k\th\vv e_r
+\vt_{(1),k}\sin k\th\vv e_\th
+\uz_{(1),k}\cos k\th\vv e_z\Bigr),\quad
P_{(1)}^{\prime}\eqdef\sum_{k=1}^\infty P_{(1),k}\cos k\th,$$
and
$$\vv u_{(1)}^{\prime\prime}\eqdef\sum_{k=1}^\infty\Bigl(\vr_{(1),k}\sin k\th\vv e_r
+\ut_{(1),k}\cos k\th\vv e_\th
+\vz_{(1),k}\sin k\th\vv e_z\Bigr),\quad
P_{(1)}^{\prime\prime}\eqdef\sum_{k=1}^\infty Q_{(1),k}\sin k\th,$$
such that $\vv u_{(1)}^\prime$ and $P_{(1)}^{\prime}$ satisfy
the following self-contained system:
\begin{equation}\label{deteu1}
\left\{
\begin{aligned}
&\pa_t\ur_{(1),k}+(\ur_{(0),0}\pa_r+\uz_{(0),0}\pa_z)\ur_{(1),k}
+(\ur_{(1),k}\pa_r+\uz_{(1),k}\pa_z)\ur_{(0),0}\\
&\qquad\qquad\qquad
-(\pa_r^2+\pa_z^2+\frac 1r\pa_r-\frac{1+k^2}{r^2})\ur_{(1),k}
+2k\f{\vt_{(1),k}}{r^2}+\pa_r P_{(1),k}=0,\\
&\pa_t\vt_{(1),k}+(\ur_{(0),0}\pa_r+\uz_{(0),0}\pa_z)\vt_{(1),k}+\f{\ur_{(0),0}\vt_{(1),k}}r\\
&\qquad\qquad\qquad
-(\pa_r^2+\pa_z^2+\frac 1r\pa_r-\frac{1+k^2}{r^2})\vt_{(1),k}
+2k\f{\ur_{(1),k}}{r^2}-\f{k}r P_{(1),k}=0,\\
&\pa_t\uz_{(1),k}+(\ur_{(0),0}\pa_r+\uz_{(0),0}\pa_z)\uz_{(1),k}
+(\ur_{(1),k}\pa_r+\uz_{(1),k}\pa_z)\uz_{(0),0}\\
&\qquad\qquad\qquad
-(\pa_r^2+\pa_z^2+\frac 1r\pa_r-\frac{k^2}{r^2})\uz_{(1),k}+\pa_z P_{(1),k}=0,\\
& \pa_r\ur_{(1),k}+\frac 1r\ur_{(1),k}+\pa_z\uz_{(1),k}
+\f{k}r\vt_{(1),k}=0,\\
& \ur_{(1),k}|_{t=0}=a^r_{k},\quad
\vt_{(1),k}|_{t=0}=b^\th_{k},\quad
\uz_{(1),k}|_{t=0}=a^z_{k},
\end{aligned}
\right.
\end{equation}
while $\vv u_{(1)}^{\prime\prime}$ and $P_{(1)}^{\prime\prime}$ satisfy
the following self-contained system:
\begin{equation}\label{deteu12}
\left\{
\begin{aligned}
&\pa_t\vr_{(1),k}+(\ur_{(0),0}\pa_r+\uz_{(0),0}\pa_z)\vr_{(1),k}
+(\vr_{(1),k}\pa_r+\vz_{(1),k}\pa_z)\ur_{(0),0}\\
&\qquad\qquad\qquad
-(\pa_r^2+\pa_z^2+\frac 1r\pa_r-\frac{1+k^2}{r^2})\vr_{(1),k}
-2k\f{\ut_{(1),k}}{r^2}+\pa_r Q_{(1),k}=0,\\
&\pa_t\ut_{(1),k}+(\ur_{(0),0}\pa_r+\uz_{(0),0}\pa_z)\ut_{(1),k}+\f{\ur_{(0),0}\ut_{(1),k}}r\\
&\qquad\qquad\qquad
-(\pa_r^2+\pa_z^2+\frac 1r\pa_r-\frac{1+k^2}{r^2})\ut_{(1),k}
-2k\f{\vr_{(1),k}}{r^2}+\f{k}r Q_{(1),k}=0,\\
&\pa_t\vz_{(1),k}+(\ur_{(0),0}\pa_r+\uz_{(0),0}\pa_z)\vz_{(1),k}
+(\vr_{(1),k}\pa_r+\vz_{(1),k}\pa_z)\uz_{(0),0}\\
&\qquad\qquad\qquad
-(\pa_r^2+\pa_z^2+\frac 1r\pa_r-\frac{k^2}{r^2})\vz_{(1),k}+\pa_z Q_{(1),k}=0,\\
& \pa_r\vr_{(1),k}+\frac 1r\vr_{(1),k}+\pa_z\vz_{(1),k}
-\f{k}r\ut_{(1),k}=0,\\
& \vr_{(1),k}|_{t=0}=0,\quad\ut_{(1),k}|_{t=0}=0,\quad\vz_{(1),k}|_{t=0}=0.
\end{aligned}
\right.
\end{equation}

Then by
taking $L^2$ inner product of \eqref{deteu12} with
$(\vr_{(1),k},\ut_{(1),k},\vz_{(1),k})$, we obtain
\begin{equation}\begin{aligned}\label{5.4}
&\f12\f{d}{dt}\|(\vr_{(1),k},\ut_{(1),k},\vz_{(1),k})\|_{L^2}^2
+\|\nabla(\vr_{(1),k},\ut_{(1),k},\vz_{(1),k})\|_{L^2}^2
+k^2\bigl\|\f{\vz_{(1),k}}r\bigr\|_{L^2}^2\\
&\quad+(1+k^2)\bigl(\bigl\|\f{\vr_{(1),k}}r\bigr\|_{L^2}^2
+\bigl\|\f{\ut_{(1),k}}r\bigr\|_{L^2}^2\bigr)
=\int_{\R^3}\Bigl(4k\f{\ut_{(1),k}\vr_{(1),k}}{r^2}
-\f{\ur_{(0),0}|\ut_{(1),k}|^2}r\\
&\qquad-(\vr_{(1),k}\pa_r+\vz_{(1),k}\pa_z)\ur_{(0),0}\cdot\vr_{(1),k}
-(\vr_{(1),k}\pa_r+\vz_{(1),k}\pa_z)\uz_{(0),0}\cdot\vz_{(1),k}\Bigr)\,dx.
\end{aligned}\end{equation}
It is worth mentioning that, not the same as the most cases
in doing estimates in PDE,
the constant in this inequality is very important. Precisely, we have
\begin{align*}
4k\Bigl|\int_{\R^3}\f{\ut_{(1),k}\vr_{(1),k}}{r^2}\,dx\Bigr|
&\leq4k\bigl\|\f{\vr_{(1),k}}r\bigr\|_{L^2}\bigl\|\f{\ut_{(1),k}}r\bigr\|_{L^2}\\
&\leq2k\bigl(\bigl\|\f{\vr_{(1),k}}r\bigr\|_{L^2}^2
+\bigl\|\f{\ut_{(1),k}}r\bigr\|_{L^2}^2\bigr)
\leq(1+k^2)\bigl(\bigl\|\f{\vr_{(1),k}}r\bigr\|_{L^2}^2
+\bigl\|\f{\ut_{(1),k}}r\bigr\|_{L^2}^2\bigr).
\end{align*}
On the other hand, by using Sobolev's embedding theorem and Young's inequality,
we get
\begin{align*}
&\Bigl|\int_{\R^3}\Bigl((\vr_{(1),k}\pa_r+\vz_{(1),k}\pa_z)\ur_{(0),0}\cdot\vr_{(1),k}
+\f{\ur_{(0),0}|\ut_{(1),k}|^2}r+(\vr_{(1),k}\pa_r+\vz_{(1),k}\pa_z)
\uz_{(0),0}\cdot\vz_{(1),k}\Bigr)\,dx\Bigr|\\
&\leq\bigl(\|\vr_{(1),k}\|_{L^4}^2+\|\vz_{(1),k}\|_{L^4}^2\bigr)
\bigl(\|\wt\nabla\ur_{(0),0}\|_{L^2}+\|\wt\nabla\uz_{(0),0}\|_{L^2}\bigr)
+\|\ut_{(1),k}\|_{L^4}^2\bigl\|\f{\ur_{(0),0}}r\bigr\|_{L^2}\\
&\leq\f12\|\nabla(\vr_{(1),k},\ut_{(1),k},\vz_{(1),k})\|_{L^2}^2
+C\|(\vr_{(1),k},\ut_{(1),k},\vz_{(1),k})\|_{L^2}^2
\|\nabla\vv u_{(0)}\|_{L^2}^4,
\end{align*}
where in the last step we have used the fact that
$\vv u_{(0)}=\ur_{(0),0}\vv e_r+\uz_{(0),0}\vv e_z$, so that
$$\|\nabla\vv u_{(0)}\|_{L^2}
\sim\|\wt\nabla\ur_{(0),0}\|_{L^2}+\|\wt\nabla\uz_{(0),0}\|_{L^2}
+\bigl\|\f{\ur_{(0),0}}r\bigr\|_{L^2}.$$
By substituting the above estimates into \eqref{5.4}, we deduce
$$\f{d}{dt}\|(\vr_{(1),k},\ut_{(1),k},\vz_{(1),k})\|_{L^2}^2
+\|\nabla(\vr_{(1),k},\ut_{(1),k},\vz_{(1),k})\|_{L^2}^2
\leq C\|(\vr_{(1),k},\ut_{(1),k},\vz_{(1),k})\|_{L^2}^2
\|\nabla\vv u_{(0)}\|_{L^2}^4.$$
Then by applying Gronwall's inequality, and using the fact that
$$\int_0^\infty\|\nabla\vv u_{(0)}(t')\|_{L^2}^4\,dt'
\leq\|\nabla\vv u_{(0)}\|_{L^\infty(\R^+;L^2)}^2
\|\nabla\vv u_{(0)}\|_{L^2(\R^+;L^2)}^2<\infty,$$
which is guaranteed by \eqref{boundu0},
and initially $\vr_{(1),k}|_{t=0}=\ut_{(1),k}|_{t=0}=\vz_{(1),k}|_{t=0}=0$,
we obtain that
\beno
\vr_{(1),k}=\ut_{(1),k}=\vz_{(1),k}=0,\quad\forall\ t>0.
\eeno
 In another word,
we can further reduce the expansion for $\vv u_{(1)}$
in \eqref{5.2} to be
\begin{equation}\label{expuvereduce}
\vv u_{(1)}=\sum_{k=1}^\infty\Bigl(\ur_{(1),k}\cos k\th\vv e_r
+\vt_{(1),k}\sin k\th\vv e_\th
+\uz_{(1),k}\cos k\th\vv e_z\Bigr),
\end{equation}
where $\ur_{(1),k},~\vt_{(1),k}$ and $\uz_{(1),k}$ are determined by
\eqref{deteu1}.

\subsection{Verification of $\ut_{(j),0}=\vr_{(j),k}=\ut_{(j),k}=\vz_{(j),k}=0$, for any $
j\in\N,~ k\in\N$}  {We shall prove this result by the induction method.}
Assume that for any $n\geq2$ and any $1\leq j\leq n-1$,
\beno
\ut_{(j),0}=\vr_{(j),k}=\ut_{(j),k}=\vz_{(j),k}=0,\quad\forall\ k\in\N,
\eeno
i.e.,
$\vv u_{(j)}$ has the following form
\begin{equation}\label{5.7}
\vv u_{(j)}=\ur_{(j),0}\vv e_r+\uz_{(j),0}\vv e_z
+\sum_{k=1}^\infty\Bigl(\ur_{(j),k}\cos k\th\vv e_r
+\vt_{(j),k}\sin k\th\vv e_\th
+\uz_{(j),k}\cos k\th\vv e_z\Bigr),
\end{equation}
then our aim is to show that $\vv u_{(n)}$ can also be written as this form, i.e.
\beno
\ut_{(n),0}=\vr_{(n),k}=\ut_{(n),k}=\vz_{(n),k}=0,\quad\forall\ k\in\N,
\eeno

Indeed, due to Remark \ref{rmk4.1}, we know that in the Euclidean coordinates,
$\vv u_{(n)}$ satisfies
\begin{equation}\label{eqtuve2}
\left\{
\begin{aligned}
&\pa_t\vv u_{(n)}+\vv u_{(0)}\cdot\nabla\vv u_{(n)}
+\vv u_{(n)}\cdot\nabla\vv u_{(0)}
-\D\vv u_{(n)}+\nabla P_{(n)}=-\sum_{j=1}^{n-1}
\vv u_{(j)}\cdot\nabla\vv u_{(n-j)},\\
& \dive\vv u_{(n)}=0,\\
&\vv u_{(n)}|_{t=0}=0.
\end{aligned}
\right.
\end{equation}
This system has zero-valued initial data. However, due to the external force term
$-\sum_{j=1}^{n-1}\vv u_{(j)}\cdot\nabla\vv u_{(n-j)}$,
the solution $\vv u_{(n)}$ in general does not vanish.

Noticing that $\vv u_{(j)}$ has the form \eqref{5.7} for $1\leq j\leq n-1$,
we can get
\begin{equation}\begin{split}\label{5.9}
\vv u_{(j)}\cdot\na\vv u_{(n-j)}=&\sum_{k_1,k_2=0}^\infty
\Bigl\{(\ur_{(j),k_1}\p_r+\uz_{(j),k_1}\p_z)
\bigl(\ur_{(n-j),k_2}\vv e_r+\uz_{(n-j),k_2}\vv e_z\bigr)
\cdot\cos k_1\th\cdot\cos k_2\th\\
&\quad+(\ur_{(j),k_1}\p_r+\uz_{(j),k}\p_z)\vt_{(n-j),k_2}\vv e_\th
\cdot\cos k_1\th\cdot\sin k_2\th\\
&\quad+\f{\vt_{(j),k_1}\ur_{(n-j),k_2}}{r}
\bigl(-\vv e_r k_2\sin k_1\th\cdot\sin k_2\th
+\vv e_\th\sin k_1\th\cdot\cos k_2\th\bigr)\\
&\quad+\f{\vt_{(j),k_1}\vt_{(n-j),k_2}}{r}\bigl(\vv e_\th k_2\sin k_1\th\cdot\cos k_2\th
-\vv e_r\sin k_1\th\cdot\sin k_2\th\bigr)\Bigr\}.
\end{split}\end{equation}
In particular, we can see that $(\vv u_{(j)}\cdot\na\vv u_{(n-j)})\cdot\vv e_r$
and $(\vv u_{(j)}\cdot\na\vv u_{(n-j)})\cdot\vv e_z$ is even in $\th$,
while $(\vv u_{(j)}\cdot\na\vv u_{(n-j)})\cdot\vv e_\th$ is odd in $\th$.
As a result, this external force
$-\sum_{j=1}^{n-1}\vv u_{(j)}\cdot\nabla\vv u_{(n-j)}$
does not appear in the equations for
$(\vr_{(n),k},\ut_{(n),k},\vz_{(n),k})$ for $k\geq 1$. Thus the equations for
$(\vr_{(n),k},\ut_{(n),k},\vz_{(n),k})$ are exactly the same as that
for $(\vr_{(1),k},\ut_{(1),k},\vz_{(1),k})$ (see \eqref{deteu12}) stated as follows:
\begin{equation*}
\left\{
\begin{aligned}
&\pa_t\vr_{(n),k}+(\ur_{(0),0}\pa_r+\uz_{(0),0}\pa_z)\vr_{(n),k}
+(\vr_{(n),k}\pa_r+\vz_{(n),k}\pa_z)\ur_{(0),0}\\
&\qquad\qquad\qquad
-(\pa_r^2+\pa_z^2+\frac 1r\pa_r-\frac{1+k^2}{r^2})\vr_{(n),k}
-2k\f{\ut_{(n),k}}{r^2}+\pa_r Q_{(n),k}=0,\\
&\pa_t\ut_{(n),k}+(\ur_{(0),0}\pa_r+\uz_{(0),0}\pa_z)\ut_{(n),k}+\f{\ur_{(0),0}\ut_{(n),k}}r\\
&\qquad\qquad\qquad
-(\pa_r^2+\pa_z^2+\frac 1r\pa_r-\frac{1+k^2}{r^2})\ut_{(n),k}
-2k\f{\vr_{(n),k}}{r^2}+\f{k}r Q_{(n),k}=0,\\
&\pa_t\vz_{(n),k}+(\ur_{(0),0}\pa_r+\uz_{(0),0}\pa_z)\vz_{(n),k}
+(\vr_{(n),k}\pa_r+\vz_{(n),k}\pa_z)\uz_{(0),0}\\
&\qquad\qquad\qquad
-(\pa_r^2+\pa_z^2+\frac 1r\pa_r-\frac{k^2}{r^2})\vz_{(n),k}+\pa_z Q_{(n),k}=0,\\
& \pa_r\vr_{(n),k}+\frac 1r\vr_{(n),k}+\pa_z\vz_{(n),k}
-\f{k}r\ut_{(n),k}=0,\\
& \vr_{(n),k}|_{t=0}=0,\quad\ut_{(n),k}|_{t=0}=0,\quad\vz_{(n),k}|_{t=0}=0.
\end{aligned}
\right.
\end{equation*}
Then the same procedure as what we have done at the end of Subsection \ref{sub5.2}
shows that for $k\geq 1$, $(\vr_{(n),k},\ut_{(n),k},\vz_{(n),k})$
indeed vanishes for all time, i.e.
\beno
\vr_{(n),k}=\ut_{(n),k}=\vz_{(n),k}=0,\quad\forall\ t>0.
\eeno

Let us turn to study the $0$-th Fourier coefficients
$\ur_{(n),0}\vv e_r+\ut_{(n),0}\vv e_\th+\uz_{(n),0}\vv e_z$.
In view of the expression \eqref{5.9}, and the following identities
\begin{align*}
&2\cos k_1\th\cdot\cos k_2\th=\cos (k_1+k_2)\th+\cos (k_1-k_2)\th,\\
&2\sin k_1\th\cdot\sin k_2\th=\cos (k_1-k_2)\th-\cos (k_1+k_2)\th,\\
&2\sin k_1\th\cdot\cos k_2\th=\sin (k_1+k_2)\th+\sin (k_1-k_2)\th,
\end{align*}
and the fact that $\vv u_{(0)}=\ur_{(0),0}\vv e_r+\uz_{(0),0}\vv e_z$
is axisymmetric without swirl, we can obtain
\beq\label{H 2}
\pa_t\ut_{(n),0}+\vv u_{(0)}\cdot\nabla\ut_{(n),0}
+\f{\ur_{(0),0}\ut_{(n),0}}r
-(\D-\frac{1}{r^2})\ut_{(n),0}=0,\quad\ut_{(n),0}|_{t=0}=0,
\eeq
and
\begin{equation}\label{5.10}
\left\{
\begin{aligned}
&\pa_t\ur_{(n),0}+\vv u_{(0)}\cdot\nabla\ur_{(n),0}
+(\ur_{(n),0}\pa_r+\uz_{(n),0}\pa_z)\ur_{(0),0}\\
&\qquad\qquad\qquad\qquad\qquad\qquad-(\D-\frac{1}{r^2})\ur_{(n),0}
+\pa_r P_{(n),0}=F^r_{(n),0},\\
&\pa_t\uz_{(n),0}+\vv u_{(0)}\cdot\nabla\uz_{(n),0}
+(\ur_{(n),0}\pa_r+\uz_{(n),0}\pa_z)\uz_{(0),0}
-\D\uz_{(n),0}+\pa_z P_{(n),0}=F^z_{(n),0},\\
& \pa_r\ur_{(n),0}+\frac 1r\ur_{(n),0}+\pa_z\uz_{(n),0}=0,\\
& \ur_{(n),0}|_{t=0}=0,\quad\uz_{(n),0}|_{t=0}=0,
\end{aligned}
\right.
\end{equation}
where $F^r_{(n),0},~F^z_{(n),0}$ are the external force terms given by
\begin{align*}
F^r_{(n),0}=-\f12&\sum_{k=1}^\infty\sum_{j=1}^{n-1}
\bigl(\ur_{(j),k}\pa_r+\uz_{(j),k}\pa_z\bigr)\ur_{(n-j),k}
+\f12\sum_{k=1}^\infty\sum_{j=1}^{n-1}\f{\vt_{(j),k}\vt_{(n-j),k}}r,\\
&F^z_{(n),0}=-\f12\sum_{k=1}^\infty\sum_{j=1}^{n-1}\bigl(\ur_{(j),k}\pa_r
+\uz_{(j),k}\pa_z\bigr)\uz_{(n-j),k}.
\end{align*}
In particular, this reflects the fact that if two profiles have same frequency,
then their product would contribute to the average of $\vv u$ in $\th$ variable.

Taking $L^2$ inner product of \eqref{H 2} with $\ut_{(n),0}$, it is easy to get that
\beno
\ut_{(n),0}=0,\quad\forall\ t>0.
\eeno

Therefore, $\vv u_{(n)}$ can also be written in the following form:
$$\vv u_{(n)}=\ur_{(n),0}\vv e_r+\uz_{(n),0}\vv e_z
+\sum_{k=1}^\infty\Bigl(\ur_{(n),k}\cos k\th\vv e_r
+\vt_{(n),k}\sin k\th\vv e_\th
+\uz_{(n),k}\cos k\th\vv e_z\Bigr).$$
This completes the proof of Theorem \ref{thm3} by induction.

\medskip

\noindent {\bf Acknowledgments.}
Y. Liu is supported by NSF of China under grant 12101053,
and the Fundamental Research Funds
for the Central Universities under Grant 310421118.
L. Xu is supported by NSF of China under grant 11671383 and 12171019.

\end{document}